\documentclass[a4paper, 10pt, twoside]{amsart}
\usepackage[utf8]{inputenc}
\usepackage{exscale, fullpage,url}
\usepackage[centertags]{amsmath}
\usepackage{color}
\usepackage{amssymb}
\usepackage{amsthm}
\usepackage{dsfont}
\usepackage[all]{xy}
\usepackage{mathrsfs}
\usepackage{upgreek}
\usepackage{comment}
\usepackage{tikz-cd}
\usepackage[hidelinks]{hyperref}
\usepackage{graphicx,fp}
\usepackage{mathtools}
\usepackage{paralist}
\usepackage{calc}
\usepackage{diagbox}
\usepackage{enumitem}
\usepackage{tikz}
\usetikzlibrary{arrows.meta,calc,patterns}

\usepackage[capitalise, noabbrev]{cleveref}
\crefname{enumi}{}{}
\usepackage{stmaryrd}

\usepackage{csquotes}
\usepackage[color=orange!10]{todonotes}

\numberwithin{equation}{section}

\usepackage{thmtools}

\newtheorem{thm}{Theorem}[section]

\newcommand{\myendsymbol}{\ensuremath{\diamondsuit}}

\declaretheorem[
  style=plain,
  title=Lemma,
  qed={},
  sharenumber=thm,
]{lem}

\declaretheorem[
  style=plain,
  title=Proposition,
  qed={},
  sharenumber=thm,
]{prop}

\declaretheorem[
  style=plain,
  title=Corollary,
  qed={},
  sharenumber=thm,
]{cor}

\declaretheorem[
  style=plain,
  title=Corollary,
  qed={\qedsymbol},
  sharenumber=thm,
]{cornoproof}

\declaretheorem[
  style=definition,
  title=Remark,
  qed={$\myendsymbol$},
  sharenumber=thm,
]{rmk}

\declaretheorem[
  style=plain,
  title=Notation,
  qed={$\myendsymbol$},
  sharenumber=thm,
]{nota}

\declaretheorem[
  style=definition,
  title=Definition,
  qed={$\myendsymbol$},
  sharenumber=thm,
]{dfn}

\declaretheorem[
  style=definition,
  title=Remark,
  qed={$\myendsymbol$},
  sharenumber=thm,
]{rem}

\newcommand{\dC}{{\mathds C}}
\newcommand{\dQ}{{\mathds Q}}
\newcommand{\dN}{{\mathds N}}

\newcommand{\dZ}{{\mathds Z}}
\newcommand{\dP}{{\mathds P}}

\newcommand{\dL}{{\mathbb L}}
\newcommand{\dA}{{\mathbb A}}

\newcommand{\cA}{\mathcal{A}}

\newcommand{\cC}{\mathcal{C}}
\newcommand{\cD}{\mathcal{D}}
\newcommand{\cE}{\mathcal{E}}
\newcommand{\cF}{\mathcal{F}}

\newcommand{\cH}{\mathcal{H}}

\newcommand{\cM}{\mathcal{M}}
\newcommand{\cN}{\mathcal{N}}
\newcommand{\cO}{\mathcal{O}}

\newcommand{\cR}{\mathcal{R}}
\newcommand{\cS}{\mathcal{S}}

\newcommand{\cU}{\mathcal{U}}

\newcommand{\fa}{\mathfrak{a}}

\newcommand{\fg}{\mathfrak{g}}

\newcommand{\fp}{\mathfrak{p}}
\newcommand{\fq}{\mathfrak{q}}

\newcommand{\fu}{\mathfrak{u}}

\newcommand{\D}{\displaystyle}

\newcommand{\SC}{\scriptstyle}

\DeclareMathOperator{\GL}{\textup{GL}}

\DeclareMathOperator{\act}{\textup{act}}

\DeclareMathOperator{\Gr}{Gr}

\DeclareMathOperator{\gr}{\textup{Gr}}

\DeclareMathOperator{\coker}{\textup{coker}}

\newcommand{\MHM}{\textup{MHM}}

\newcommand{\mbc}{\mathds{C}}

\newcommand{\mbp}{\mathds{P}}
\newcommand{\mbq}{\mathds{Q}}
\newcommand{\mbqH}{\mathds{Q}^{\textup{H}}}

\newcommand{\mcf}{\mathcal{F}}

\newcommand{\mch}{\mathcal{H}}
\newcommand{\mchp}{\mathcal{H}}

\newcommand{\mco}{\mathcal{O}}

\newcommand{\ra}{\rightarrow}
\newcommand{\lra}{\longrightarrow}

\newcommand{\KH}{\mathit{KH}}
\newcommand{\PH}{\mathit{PH}}
\newcommand{\CC}{{\mathbb{C}}}
\newsavebox\foobox

\newcommand{\suchthat}{\;\ifnum\currentgrouptype=16 \middle\fi|\;}

\DeclareMathOperator{\codim}{codim}

\DeclareMathOperator{\trace}{trace}

\DeclareMathOperator{\rk}{rk}

\DeclareMathOperator{\sheafHom}{\mathscr{H}\kern -3pt\textit{om}\kern 1pt}
\DeclareMathOperator{\sheafTor}{\mathscr{T}\kern -3pt\textit{or}\kern 1pt}
\DeclareMathOperator{\sheafExt}{\mathscr{E}\kern -2pt\textit{xt}\kern 1pt}
\DeclareMathOperator{\sheafDer}{\mathscr{D}\kern -1pt\textit{er}\kern 1pt}
\DeclareMathOperator{\differential}{d\!}

\DeclareMathOperator{\Ad}{Ad}

\newcommand{\N}{\mathds{N}}

\renewcommand{\P}{\mathds{P}}
\newcommand{\R}{\mathds{R}}
\newcommand{\C}{\mathds{C}}
\newcommand{\Q}{\mathds{Q}}
\newcommand{\Z}{\mathds{Z}}

\renewcommand{\O}{\mathcal{O}}
\newcommand{\Ell}{\mathscr{L}}

\let\originalleft\left
\let\originalright\right
\renewcommand{\left}{\mathopen{}\mathclose\bgroup\originalleft}
\renewcommand{\right}{\aftergroup\egroup\originalright}

\newcommand\restr[2]{{#1_{| {#2}}}}

\begin{document}

\title{Tautological systems  and local cohomology}

\author[P.~G\"orlach]{Paul G\"orlach}
\address{\linebreak
P.~G\"orlach\\
Otto-von-Guericke-Universität Magdeburg\\
Fakult\"at f\"ur Mathematik\\
Institut f\"ur Algebra und Geometrie\\
Universitätsplatz 2\\
39106 Magdeburg\\
Germany}
\email{\href{mailto:paul.goerlach@ovgu.de}{paul.goerlach@ovgu.de}}

\author[T.~Reichelt]{Thomas Reichelt}
\address{\linebreak
T.~Reichelt\\
Lehrstuhl f\"ur Mathematik VI \\
Institut f\"ur Mathematik\\
Universit\"at Mannheim,
A 5, 6 \\
68131 Mannheim\\
Germany}
\email{\href{mailto:thomas.reichelt@math.uni-mannheim.de}{thomas.reichelt@math.uni-mannheim.de}}

\author[C.~Sevenheck]{Christian Sevenheck}
\address{\linebreak
C.~Sevenheck\\
Fakult\"at f\"ur Mathematik\\
Technische Universit\"at Chemnitz\\
09107 Chemnitz\\
Germany}
\email{\href{mailto:christian.sevenheck@mathematik.tu-chemnitz.de}{christian.sevenheck@mathematik.tu-chemnitz.de}}

\author[U.~Walther]{Uli Walther}
\address{\linebreak
  U.~Walther\\
  Department of Mathematics\\
  Purdue University\\
  West Lafayette, IN 47907\\
  USA
}
\email{\href{mailto:walther@purdue.edu}{walther@purdue.edu}}

\thanks{PG and CS were partially supported by DFG grant SE 1114/6-2.
 UW was supported in part by NSF grant DMS-2100288 and by
 Simons Foundation Collaboration Grant for Mathematicians \#580839 and SFI-MPS-TSM-00012928.}



\begin{abstract}
  \noindent  We discuss the connections between tautological systems and  the local cohomology of cones over homogeneous spaces. We study a derived version of tautological systems, related to the Chevalley--Eilenberg complex, and show that in many cases it underlies a complex of mixed Hodge modules.
\end{abstract}

\maketitle

\tableofcontents

\section{Introduction} \label{sec:intro}

The purpose of this paper is to investigate the relationship between the local cohomology modules of cones over smooth projective varieties, and certain equivariant  differential systems in the case where these smooth varieties are homogeneous spaces. This continues our study of such $\cD$-modules in \cite{GRSSW} and \cite{GS-Duality}.

Let
\[
\rho:G\rightarrow \GL(V)
\]
be a representation of a connected linear algebraic group $G$, and choose a $G$-variety
\[
\overline{Y}\subseteq V,
\]
which typically will be the closure of a $G$-orbit $Y\subseteq V$. For a  Lie algebra homomorphism
\[
\beta:\fg\rightarrow \dC,
\]
we have studied in \cite{GRSSW} the \emph{tautological system} $\tau(\rho, \overline{Y}, \beta)$, an equivariant $\cD$-module on the dual space $V^\vee$ that arises as a Fourier--Laplace transform of a module $\hat{\tau}(\rho, \overline{Y}, \beta)\in\textup{Mod}(\cD_V)$. Later,  in \cite{GS-Duality} we investigated the duality theory of these systems.

It is frequently convenient to express $\hat{\tau}(\rho, \overline{Y}, \beta)$ as the terminal cohomology of a complex of $\cD_V$-modules that generalizes both the Chevalley--Eilenberg complex in Lie theory and the Euler--Koszul complex in the theory of hypergeometric $\cD$-modules. In the present paper, we study the relationhip of this complex with local cohomology modules, specifically for the case where $\overline{Y}$ is the affine cone over a projective homogeneous space.

\medskip

In order to formulate our main result, let us set up some notation. We concentrate here on the case we are mostly interested in (that is, cones over projective homogeneous spaces), some of our results in the main text are valid in a more general context.

Let
\[
X\coloneqq G_0/P_0,
\]
where $G_0$ is a connected linear algebraic group, and $P_0$ is a parabolic subgroup. Let $\mathscr{L}$ be a very ample $G_0$-equivariant line bundle on $X$, and consider the embedding $X\hookrightarrow \dP V$ given by $|\mathscr{L}|$. Write $\hat{X}\subseteq V$ for the affine cone of $X$. Then $\hat{X}$ is a $G$-space, where
\[
\rho\colon G=\dC^*\times G_0\longrightarrow\GL(V)
\]
acts naturally on $V$ (with $\dC^*$ acting via the usual scaling). Let
\[
\fg\coloneqq\dC\oplus \fg_0
\]
be the Lie algebra of $G$ and let $\cU(\fg)$
be its universal enveloping algebra. Denote by
\[
Z_V \colon \fg\rightarrow \Theta_V
\]
the natural map given by the derivative of the action $\rho$ and write
$\cA_V \coloneqq \O_V \otimes \cU(\mathfrak g)$ for the associative $\C$-algebra with multiplication determined by
\[
\xi \cdot f = f \cdot \xi + Z_V(\xi)(f) \qquad \text{for $f \in \O_V$, $\xi \in \mathfrak g$}.
\]
The map $Z_V$ extends to a homomorphism $\cA_V\rightarrow\cD_V$ of $\C$-algebras.
In this setup, we consider the following object.
\begin{dfn}
Let $\beta:\fg\rightarrow \dC$ be a Lie algebra homomorphism. The \emph{Chevalley--Eilenberg--Euler--Koszul complex} associated to $X$, $\Ell$ and $\beta$ is
\[
  \hat T(\rho, \hat{X}, \beta) \coloneqq
  \cD_V\otimes^\dL_{\cA_V} i_*\cO_{\hat{X}}\{\beta'\} \in D^b(\mathcal{D}_V),
\]
where $i \colon \hat X \to V$ is the inclusion and
\[
i_*\cO_{\hat{X}}\{\beta'\} \coloneqq i_*\cO_{\hat{X}} \otimes_{\cO_V} \cA_V/\cA_V(\xi-(\trace\circ \differential\rho-\beta)(\xi) \mid \xi \in \mathfrak g).
\]
Notice that there is an explicit presentation of $\hat T(\rho,\hat X,\beta)$ with a familiar differential from Lie theory (see, e.g.\ the discussion in \cite[Section 2.1]{GS-Duality}).
\end{dfn}

We present in this article a functorial interpretation of the Chevalley--Eilenberg--Euler--Koszul complex, some Hodge theoretic consequences,  and a relation to the local cohomology of the cone $\hat{X}$. We summarize some of these findings as follows.  Denote by
\[
\iota\colon \hat{X}\backslash\{0\} \hookrightarrow \hat X\hookrightarrow V
\]
the canonical locally closed embedding. Moreover, we denote, for any complex algebraic variety $Z$, by $a\colon Z\rightarrow \{pt\}$ the structure map of $Z$ and we write
$d_{Z}\coloneqq \dim_\CC(Z)$. We also refer to the summary of notations on algebraic $\cD$-modules and Hodge modules in our previous paper \cite[page 525-526]{GRSSW}.
Then the following holds.
\begin{thm}[{see \cref{theo:ColocFunct}, \cref{cor:ThatHodge} and \cref{cor:weightLocCohom}}]\label{thm:MainTheo}
If $d_X>0$ and  $\hat{X}$ is Gorenstein, then we have isomorphisms in $D^b_h(\cD_V)$:
\[
\hat{T}(\rho,\hat{X},\gamma) \cong
\iota_\dag \cO_Y \otimes_\dC a_\dag \cO_{P_0}[d_{P_0}]
\qquad\textup{and}\qquad
\hat{T}(\rho,\hat{X},0)  \cong \iota_+ \cO_Y \otimes_\C a_+ \cO_{P_0},
\]
where $\gamma$ is the Lie algebra homomorphism defined in \cref{rem:GorensteinGamma} below. Consequently, $\hat{T}(\rho,\hat{X},\gamma)$ and $\hat{T}(\rho,\hat{X},0)$ underlie objects in $D^b\MHM(V)$.

The cohomology modules of the Chevalley--Eilenberg--Euler--Koszul complex can be expressed in terms of the local cohomology along $\hat{X}$ and the singular cohomology of $P_0$ as
\[
H^{-k} \hat T(\rho, \hat X, 0) \cong \left(H_{\{0\}}^{d_V} \O_V \otimes_\C H^{d_{P_0}-d_{\hat X}+1-k}(P_0,\C)\right) \oplus
\bigoplus_{i+j=d_{P_0}-k} H^{d_V-d_{\hat X}+i}_{\hat X} \O_V \otimes_\C H^j(P_0,\C).
\]

Moreover,
the Hodge module corresponding to the first local cohomology group $H^{d_V-d_{\hat{X}}}_{\hat{X}}\cO_V$ has weights in $\{2d_V-d_{\hat{X}},\ 2d_V-d_{\hat{X}}+1\}$, and the (Hodge modules of the) higher local cohomology groups $H^{d_V-d_{\hat{X}}+i}_{\hat{X}}\cO_V$ are all pure of weight $2d_V-d_{\hat{X}}+i+1$ for $i\in\{1,\ldots,d_{\hat{X}}-2\}$.
\end{thm}

\medskip

A short overview on the different parts of this paper follows. In  \cref{sec:LocCohom} we discuss quite generally the local cohomology modules of cones over an arbitrary smooth projective varieties $X$. We relate them to the cohomology of $X$, and use this to determine their weight filtrations. In \cref{sec:Coloc}, turning to an equivariant setting, we recall the Chevalley--Eilenberg--Euler--Koszul complex and generalize in \cref{prop:FunctConstrHatT},  for cones over homogeneous spaces, the colocalization property  proved originally in \cite{GRSSW} for the terminal cohomology module only. Finally, in \cref{sec:Orbit} we study the restriction of this complex to the open orbit (e.g.\ the complement of the origin of the affine cone over the given projective homogeneous space), and as a consequence, obtain the identification mentioned in our main theorem above as well as its Hodge theoretic consequences (see \cref{theo:ColocFunct} and \cref{cor:ThatHodge}). As a  further consequence, using also the results from \cref{sec:LocCohom}, we give the interpretation  mentioned above of the Chevalley--Eilenberg--Euler--Koszul complex in terms of the local cohomology groups of the affine cone (see \cref{cor-lcdef}).

\section{Local cohomology of cones and weight filtration}
\label{sec:LocCohom}

As a preliminary consideration, we develop in this section a description of the local cohomology groups of cones. To this end, let $X$ be a connected smooth projective variety of positive dimension, and fix a very ample line bundle $L$ on $X$. Then we have an embedding  $X \hookrightarrow \mbp(V)$ where $V = H^0(X,L)^\vee$. Let $\hat{X}\subseteq V$ be the affine cone of $X$ in $V$ and denote by $Y\coloneqq \hat{X} \setminus \{0\}$ the punctured cone.

Consider the maps $Y\overset{\overline{j}} \hookrightarrow \hat{X} \overset{i_0}\hookleftarrow \{0\}$. There is the following adjunction triangle in $D^b\MHM(\hat{X})$
\[
i_{0!} \mbqH_{\{0\}} \lra \omega_{\hat{X}} \lra \overline{j}_* \omega_{Y} \overset{+1}\lra
\]
Let $k: \hat{X} \ra V$ be the canonical closed embedding and define $\iota_0 \coloneqq (k \circ i_0): \{0\} \ra V$ as well as $\iota = (k \circ \overline{j}) : Y \ra V$. Applying $k_!= k_*$ to the triangle above and using $\omega_{\hat{X}} = k^! \omega_{V}$ we get the following triangle in $D^b\MHM(V)$
\[
\iota_{0!} \mbqH_{\{0\}} \lra k_! k^! \omega_{V} \lra \iota_* \omega_{Y} \overset{+1}\lra
\]
Since $V$ and $Y$ are smooth,
\[
\omega_V = \mbqH_V(d_V)[2d_V] \qquad \text{and} \qquad \omega_Y = \mbqH_{Y}(d_Y)[2d_{Y}],
\]
and we obtain the triangle
\begin{equation}\label{eq:TopSeqWithTate}
\iota_{0!} \mbqH_{\{0\}} \lra k_! k^! \mbqH_V(d_V)[2d_V] \lra \iota_* \mbqH_Y(d_Y)[2d_{Y}] \overset{+1}\lra.
\end{equation}
The underlying complexes of $\cD$-modules fit into the following triangle
\[
\iota_{0+} \mco_{\{0\}} \lra R\Gamma_{\hat{X}}\mco_V[d_V] \lra \iota_+ \mco_Y[d_Y] \overset{+1}\lra
\]
with corresponding long exact cohomology sequence
\[
\ldots \lra H^\ell(\iota_{0+} \mco_{\{0\}}) \lra H^{\ell+d_V}_{\widehat{X}} \mco_V \lra H^{\ell+d_Y} \iota_+ \mco_Y \lra \ldots
\]
Since $H^\ell(\iota_{0+} \mco_{\{0\}}) = 0$ for $\ell\neq 0$ we get the isomorphisms
\[
H^k \iota_+ \mco_Y \simeq H^{k+d_V-d_Y}_{\widehat{X}} \mco_V \qquad \text{for}\quad k \neq d_Y-1, d_Y
\]
as well as the following exact sequence of $\mathcal{D}_V$-modules
\[
0 \lra H^{d_V-1}_{\widehat{X}} \mco_V  \lra H^{d_Y-1} \iota_+ \mco_Y \lra H^0(\iota_{0+} \mco_{\{0\}}) \lra H^{d_V}_{\widehat{X}} \mco_V \lra H^{d_Y} \iota_+ \mco_Y \lra 0.
\]

We have $H^{d_V}_{\widehat{X}} \mco_V = 0$ by the Hartshorne--Lichtenbaum Theorem \cite{Hartshorne-CDAV,PeskineSzpiro}, and, since $\hat X$ is irreducible of dimension 2 or more, $H^{d_V-1}_{\widehat{X}} \mco_V = 0$ by  the Second Vanishing Theorem, due in various shapes to Hartshorne, Ogus, Peskine--Szpiro, Huneke--Lyubeznik, Zhang, and Batavia \cite{Hartshorne-CDAV, Ogus-LCDAV, PeskineSzpiro,HunekeLyubeznik,Zhang-second, Batavia}. It follows that
\[
 H^{d_Y-1} \iota_+ \mco_Y \simeq H^0(\iota_{0+} \mco_{\{0\}}) \simeq H^{d_V}_{\{0\}} \mco_V \qquad \text{and} \qquad H^{d_Y} \iota_+ \mco_Y = 0.
\]
Since $\iota_+$ is left exact we finally obtain the following result.
\begin{prop}\label{prop:LocCohomCone}
Under the above assumptions, we have isomorphisms in $\textup{Mod}_h(\mathcal{D}_V)$\[
H^k \iota_+ \mco_Y \simeq
\begin{cases}
H^{k+d_V-d_Y}_{\widehat{X}} \mco_V & \textup{for } k \in [0, d_Y-2];\\
H^{d_V}_{\{0\}} \mco_V & \textup{for } k =d_Y-1;\\ 0 & \textup{otherwise}.
\end{cases}
\]
\end{prop}
Notice that due to the Tate twists occuring in
the sequence \eqref{eq:TopSeqWithTate},
the weight filtrations on the two sides of the  isomorphism of Hodge modules corresponding to  the previous proposition are shifted by $2(d_V-d_Y)$.
In the remainder of this section, we will study the cohomology of $\iota_+ \cO_Y$ and its dual $\iota_\dag \cO_Y$ from a topological point of view. In particular, we will determine their weight filtrations. Most interesting is of course the 0-th cohomology modules of both complexes, since one corresponds, by \cref{prop:LocCohomCone}, to the local cohomology group of $\hat{X}$ of smallest degree and the other to the dual thereof. Notice that the recent paper \cite{SabbahGarciaLopez} (see, especially Example 1 in loc.cit.) contains related results, also incorporating the Hodge filtration, by studying the modules $H^k_{\{0\}}(H^\ell_{\hat{X}}\cO_V)$.

\medskip

To start with, consider the dual bundle $L^\vee$ with projection $p^\vee:L^\vee\ra X$ and let
\[
L^{\vee, \ast} \coloneqq L^\vee \setminus \{ \text{zero section}\}.
\]
We have the following diagram with Cartesian squares
\[
\begin{tikzcd}[column sep=huge, row sep=large]
L^{\vee,\ast} \arrow[d, "\simeq"'] \arrow[r, "j"] &
L^\vee \arrow[d, "\pi"] \arrow[r, bend right=35, "p^\vee"'] &
X \arrow[d, "a"] \arrow[l, "i"'] \\
\hat{X}\setminus\{0\} \arrow[r, "\overline{j}"'] &
\hat{X} &
\{0\} \arrow[l, "i_0"]
\end{tikzcd}
\]

From the adjunction triangle
\begin{equation}\label{eq:adj_tri}
i_! i^! (p^\vee)^! \mbqH_X \lra (p^\vee)^! \mbqH_X \lra  j_* j^*(p^\vee)^! \mbqH_X \overset{+1}{\lra}
\end{equation}
we get, in light of the fact that $(p^\vee)^! = (p^\vee)^*(1)[2]$, the triangle
\[
i_!\mbqH_X \lra \mbqH_{L^\vee}(1)[2] \lra  j_* j^* \mbqH_{L^\vee}(1)[2] \overset{+1}{\lra}
\]

Writing ${^p}\mbqH_X = \mbqH_X[d_X]$ (and, \emph{mutatis mutandis},  for other spaces) and applying  a Tate twist, we derive the exact sequence
\begin{equation}\label{eq:adj_exseq}
0 \lra {^p}\mbqH_{L^\vee} \lra  j_* j^* {^p}\mbqH_{L^\vee} \lra i_!{^p}\mbqH_X(-1) \lra 0
\end{equation}
of mixed Hodge modules.
Apply the functor $\pi_! = \pi_*$ and consider the resulting long exact cohomology sequence
\begin{align}\label{eq:longExPerv}
\ldots \ra \mchp^{k}\pi_! {^p}\mbqH_{L^\vee} \ra \mchp^{k}\pi_! j_* {^p}\mbqH_{L^{\vee,*}} \ra  \mchp^{k}\pi_! i_! {^p} \mbqH_X (-1) \ra \ldots
\end{align}
yields the following result.
\begin{lem}\label{lem:SuppAtZero}
\begin{enumerate}
    \item $\mchp^\ell \pi_! i_! {^p}\mbqH_X \simeq i_{0!} H^{\ell+d_X}(X,\mbq)$
        \item $\mchp^\ell \pi_! j_* {^p}\mbqH_{L^{\vee,*}} \simeq \begin{cases}0 & \text{for } \ell \leq -1\\  i_{0!}H^{\ell+d_X+1}(L^{\vee,*},\mbq) & \text{for } \ell \geq 1\end{cases}$.

    \item $\mch^\ell \pi_!{^p}\mbqH_{L^\vee} \simeq  \begin{cases}  i_{0!} H^{\ell+d_X-1}(X,\mbq)(-1) & \text{for } \ell \leq -1 \\i_{0!} H^{\ell+d_X+1}(X,\mbq) & \text{for } \ell \geq 1.\end{cases}$
\end{enumerate}
\end{lem}
\begin{proof}
1.) This isomorphism follows from the commutativity of the right square in the diagram above:
\begin{equation}\label{eq:IdCohomX}
\mchp^\ell \pi_! i_! {^p}\mbqH_X \simeq \mchp^\ell i_{0!} a_! {^p}\mbqH_X  \simeq i_{0!} H^{\ell+d_X}(X,\mbq)
\end{equation}

2.) We first show that $\mchp^\ell \pi_! j_* {^p}\mbqH_{L^{\vee,*}}$ has support at $\{0\}$ for $\ell \neq 0$. This follows from the isomorphism
\begin{equation}\label{eq:SuppOnZero}
\overline{j}^*\mchp^\ell \pi_! j_* {^p}\mbqH_{L^{\vee,*}} \simeq \overline{j}^*\mchp^\ell \overline{j}_* {^p}\mbqH_{\hat{X} \setminus \{0\}} = 0
\end{equation}
for $\ell \neq 0$.

The claim for $\ell \leq -1$ follows from
\begin{equation}\label{eq:1}
\mchp^\ell \pi_! j_* {^p}\mbqH_{L^{\vee,*}} \simeq \mchp^\ell \overline{j}_* {^p}\mbqH_{\hat{X} \setminus \{0\}} =0 \quad \text{for } \ell \leq -1
\end{equation}
since $\overline{j}_*$ is left $t$-exact (see \cite[Theorem 5.2.4 (iii)]{Dimca}).

For $\ell \geq 1$ we consider the spectral sequence
\begin{eqnarray*}
E_{k,\ell}^2 = \mch^k\,  a_* \mchp^\ell \pi_!j_*{^p}\mbqH_{L^{\vee,*}} &\Longrightarrow&E^\infty_{k+\ell} = \mch^{k + \ell}\, a_* \pi_! j_* {^p} \mbqH_{L^{\vee,*}} = \mch^{k+\ell} a_* \mbqH_{L^{\vee,*}}[d_X+1]\\
&&\phantom{E^\infty_{k+\ell}}=H^{k+\ell+d_X+1}(L^{\vee,*},\mbq).
\end{eqnarray*}

We have $E^2_{k, \ell} = 0 $ if $k \neq 0 \neq \ell $ (by \cref{eq:SuppOnZero} and since $a_*$ is exact on objects with support in $\{0\}$) and also if $k \geq 0$ (since $a_*$ is right $t$-exact by \cite[Theorem 5.2.16 (i)]{Dimca}).
We deduce that
\[
\mch^0\,  a_* \mchp^\ell \pi_!j_* {^p}\mbqH_{L^{\vee,*}} \simeq E^2_{0,\ell} \simeq E^{\infty}_\ell \simeq H^{\ell+d_X+1}(L^{\vee,*},\mbq) \quad \text{for }\ell \geq 0,
\]
and in particular,
\[
\mch^\ell \pi_! j_*{^p}\mbqH_{L^{\vee,*}} \simeq  i_{0!} H^{\ell+d_X+1}(L^{\vee,*},\mbq) \quad \text{ for } \ell \geq 1.
\]

3.) We first show that $\mchp^\ell \pi_! {^p} \mbqH_{L^\vee}$ has support at $\{0\}$ for $\ell \neq 0$.
This follows from
\[
\overline{j}^* \mchp^\ell \pi_! {^p} \mbqH_{L^\vee} \simeq  \mchp^\ell  \overline{j}^*\pi_! {^p} \mbqH_{L^\vee} \simeq \mchp^\ell  j^* {^p} \mbqH_{L^{\vee}}  \simeq \mchp^\ell\, {^p}\mbqH_{L^{\vee}} = 0
\]
for $ \ell \neq 0$.

For $\ell \geq 1$ we consider the spectral sequence

\begin{eqnarray}
E_{k,\ell}^2 = \mch^k\,  a_* \mchp^\ell \pi_!{^p}\mbqH_{L^\vee}&\Longrightarrow& E^\infty_{k+\ell} = \mch^{k + \ell}\, a_* \pi_! {^p} \mbqH_{L^\vee} = \mch^{k+\ell} a_* \mbqH_{L^\vee}[d_X+1] \notag\\
&&\phantom{E^\infty_{k+\ell}}= H^{k+\ell+d_X+1}(L^\vee,\mbq) = H^{k+\ell+d_X+1}(X,\mbq) \label{eq:specseq1}.
\end{eqnarray}

Now we have $E^2_{k, \ell} = 0 $ if $k \neq 0 \neq \ell $ (by \cref{lem:SuppAtZero} part 3, and since $a_*$ is exact on objects with support in $\{0\}$) and also  if $k \geq 0$ (since $a_*$ is right $t$-exact by \cite[Theorem 5.2.16 (i)]{Dimca})

We deduce that
\[
\mch^0\,  a_* \mchp^\ell \pi_!{^p}\mbqH_{L^{\vee,*}} \simeq E^2_{0,\ell} \simeq E^{\infty}_\ell \simeq H^{\ell+d_X+1}(L^{\vee},\mbq) \simeq H^{\ell+d_X+1}(X,\mbq)  \quad \text{for }\ell \geq 0,
\]
and so
\[
\mch^\ell \pi_!{^p}\mbqH_{L^{\vee,*}} \simeq  i_{0!} H^{\ell+d_X+1}(X,\mbq) \quad \text{ for } \ell \geq 1.
\]

The Grothendieck spectral sequence
\begin{align*}
E_{k,\ell}^2 &= \mch^k\,  a_{!} \mchp^\ell \pi_!{^p}\mbqH_{L^\vee} \Longrightarrow E^\infty_{k+\ell} =  \mch^{k + \ell} a_! \pi_! {^p}\mbqH_{L^\vee}
\end{align*}
gives
\begin{align*}
\mch^0\,  a_{!} \mchp^\ell \pi_!{^p}\mbqH_{L^\vee} &\simeq E^2_{0,\ell} \simeq E^{\infty}_\ell =  \mch^{k + \ell} a_! \pi_! {^p}\mbqH_{L^\vee} \simeq \mch^{k + \ell} a_! p^\vee_! {^p}\mbqH_{L^\vee}  \simeq  H^{k+\ell+ d_X-1}(X,\mbq)(-1) \quad \text{for } \ell \leq 0,
\end{align*}
where we have used $p^\vee_! \mbqH_{L^\vee} \simeq \mbqH_X(-1)[-2]$.
In particular,
\[
\mch^\ell \pi_!{^p}\mbqH_{L^\vee} \simeq  i_{0!} H^{\ell+d_X-1}(X,\mbq)(-1) \quad \text{ for } \ell \leq -1.
\]
\end{proof}
As a consequence of \cref{lem:SuppAtZero},  the long exact cohomology sequence\eqref{eq:longExPerv} becomes
\begin{equation}\label{eq:FirstLongExSeqFinal}
\xymatrix{
 \ar[r] & i_{0!} H^{d_X-3}(X,\mbq)(-1) \ar[r] & 0 \ar[r] & i_{0!} H^{d_X-2}(X,\mbq)(-1)\\
\ar[r] & i_{0!} H^{d_X-2}(X,\mbq)(-1) \ar[r] &
0
\ar[r] & i_{0!} H^{d_X-1}(X,\mbq)(-1)\\
 \ar[r] &\mchp^0 \pi_* {^p}\mbqH_{L^\vee} \ar^{\alpha}[r]&   \mchp^0 \pi_* j_* {^p}\mbqH_{L^{\vee,*}} \ar[r] & i_{0!} H^{d_X}(X,\mbq)(-1) \\
 \ar[r] & i_{0!} H^{d_X+2}(X,\mbq) \ar[r] &
i_{0!} H^{d_X+2}(L^{\vee,*},\mbq)   \ar[r]& i_{0!} H^{d_X+1}(X,\mbq)(-1) \\
\ar[r] & i_{0!} H^{d_X+3}(X,\mbq) \ar[r] &
i_{0!} H^{d_X+3}(L^{\vee,*},\mbq) \ar[r] & i_{0!} H^{d_X+2}(X,\mbq)(-1) \\
\ar[r] & \ldots\\}
\end{equation}

In order to identify the connecting homomorphism $i_{0!} H^{d_X + \ell}(X,\mbq)(-1) \ra i_{0!} H^{d_X + \ell+2}(X,\mbq)$ in the above sequence \eqref{eq:FirstLongExSeqFinal} for $\ell \geq 0$ we remark that taking global sections of the triangle \eqref{eq:adj_exseq} we get the Thom--Gysin sequence (up to a Tate-twist), see, e.g., \cite[Section 1.3]{RSW}:
\[
\lra H^k(X,\mbq)(-1) \lra H^{k +2}(X,\mbq) \lra H^{k+2}(L^{\vee,*},\mbq) \overset{+1}\lra
\]
The connecting homomorphism $H^{d_X + \ell} (X,\mbq)(-1) \lra H^{d_X + \ell +2}(X,\mbq)$ is induced by the map between the global sections of ${^p} \mch^\ell \pi_!i_! {^p}\mbqH_X$ (cf. \eqref{eq:IdCohomX}) and the $E^\infty_{d_X +\ell+2}$-term  of the spectral sequence \eqref{eq:specseq1}. However, the $E^\infty_{d_X +\ell+2}$-term is equal to the $E^2_{0,\ell+1}$-term of the spectral sequence, which in turn is equal to the global sections of $i_{0!} H^{d_X + \ell+2}(X,\mbq) \simeq \mchp^{\ell+1} \pi_!{^p}\mbqH_{L^\vee}$. We conclude that
the connecting homomorphism in the lower part of the  sequence \eqref{eq:FirstLongExSeqFinal} is induced by the cup product with the first Chern class of $L^\vee$.

Starting with the adjunction triangle
\begin{equation}
j_! j^* (p^\vee)^* \mbqH_X \lra (p^\vee)^* \mbqH_X \lra  i_! i^*(p^\vee)^* \mbqH_X \overset{+1}{\lra}
\end{equation}
we get by duality (and a Tate twist) the long exact sequence
\begin{equation}\label{eq:SecondLongExSeqFinal}
\xymatrix{
 \ar[r] & i_{0!} H^{d_X-2}(X,\mbq) \ar[r] & \mchp^{-2}\pi_! j_! {^p}\mbqH_{L^{\vee,*}} \ar[r] & i_{0!} H^{d_X-3}(X,\mbq)(-1)\\
\ar[r] & i_{0!} H^{d_X-1}(X,\mbq) \ar[r] & \mchp^{-1}\pi_! j_! {^p}\mbqH_{L^{\vee,*}} \ar[r] & i_{0!} H^{d_X-2}(X,\mbq)(-1)\\
 \ar[r] &i_{0!} H^{d_X}(X,\mbq) \ar[r]&   \mchp^0 \pi_! j_! {^p}\mbqH_{L^{\vee,*}} \ar^{\beta}[r] & \mchp^0 \pi_! {^p}\mbqH_{L^\vee}\\
 \ar[r] & i_{0!} H^{d_X+1}(X,\mbq) \ar[r] & 0  \ar[r]& i_{0!} H^{d_X+2}(X,\mbq) \\
\ar[r] & i_{0!} H^{d_X+2}(X,\mbq) \ar[r] & 0  \ar[r] & i_{0!} H^{d_X+3}(X,\mbq) }
\end{equation}
where the connecting homomorphism is again given by cup product with the Euler class.

\begin{lem}\label{lem:decomp} There is the following isomorphism in $\MHM(\hat{X})$:\[
\mchp^0 \pi_* {^p}\mbqH_{L^\vee} \simeq \overline{j}_{!*} {^p}\mbqH_{\hat{X}\backslash \{0\}}\oplus i_{0!} H^{d_X+1}(X,\mbq),
\]
where, as usual, for a $\cM\in\MHM(U)$ and for an embedding $\gamma:U\rightarrow \overline{U}$, we write $\gamma_{!*}\cM$ for the image in the abelian category $\MHM(\overline{U})$ of the morphism $\cH^0 \gamma_!\cM \rightarrow \cH^0 \gamma_*\cM$.
\end{lem}
\begin{proof}
Notice that decomposition theorem applied to  the pure Hodge module $\mchp^0 \pi_* {^p}\mbqH_{L^\vee}= \mchp^0 \pi_! {^p}\mbqH_{L^\vee}$ yields
$$
\mchp^0 \pi_* {^p}\mbqH_{L^\vee} = \overline{j}_{!*} {^p}\mbqH_{\hat{X}\backslash \{0\}}\oplus \cF_0
$$
where $\textup{supp}(\cF_0)=\{0\}$.

There is the Grothendieck spectral sequence
\begin{eqnarray*}
E_{k,\ell}^2 = \mch^k\,  i_0^* \mchp^\ell \pi_!{^p}\mbqH_{L^\vee}&\Longrightarrow &
E^\infty_{k+\ell} = \mch^{k + \ell}\, i_0^* \pi_! {^p} \mbqH_{L^\vee}
= \mch^{k+\ell} a_! i^*  {^p} \mbqH_{L^\vee}
\\
&&\phantom{E^\infty_{k+\ell}}
=\mch^{k+\ell} a_! i^*   \mbqH_{L^\vee}[d_X+1] = \mch^{k+\ell} a_! \mbqH_{X}[d_X+1]
\\
&&\phantom{E^\infty_{k+\ell}}= \mch^{k+\ell} a_* \mbqH_{X}[d_X+1] = H^{k+\ell+d_X+1}(X,\mbq).
\end{eqnarray*}
This time, $E^2_{k, \ell} = 0 $ if $k \neq 0 \neq \ell $ (by the proof of \cref{lem:SuppAtZero} part 3, and since $i_0^*$ is $t$-exact on objects with support in $\{0\}$) and also  if $k \geq 0$ (since $i_0^*$ is right $t$-exact by \cite[Theorem 5.2.4 (iii)]{Dimca}).

We deduce
\[
\mch^0\,  i_0^* \mchp^\ell \pi_!{^p}\mbqH_{L^\vee} = E^2_{0,\ell} = E^{\infty}_\ell \simeq H^{\ell+d_X+1}(X,\mbq) \quad \text{for }\ell \geq 0.
\]
Since $\mch^0 i_0^* \overline{j}_{!*} {^p}\mbqH_{\hat{X}\backslash \{0\}} = 0$ by \cite[1.4.24]{BBD82} we get
$\mch^0 i_0^* \mcf_0 \simeq H^{d_X+1}(X,\mbq)$, which shows the claim.

\end{proof}
\cref{lem:decomp} shows that $\overline{j}_{!*} {^p}\mbqH_{\hat{X}\backslash \{0\}}$ is necessarily the image of the morphism
\[
\mchp^0 \pi_! j_! {^p}\mbqH_{L^{\vee,*}} \longrightarrow  \mchp^0 \pi_! {^p}\mbqH_{L^\vee}.
\]

Similarly, the image of the morphism
$ \mchp^0 \pi_* {^p}\mbqH_{L^\vee} \longrightarrow \mchp^0 \pi_* j_* {^p}\mbqH_{L^{\vee,*}}$ in the exact sequence \eqref{eq:FirstLongExSeqFinal} is $\overline{j}_{!*} {^p}\mbqH_{\hat{X}\backslash \{0\}}$, and therefore  its kernel is isomorphic to $i_{0!} H^{d_X-1}(X,\mbq)$. This is consistent with the isomorphism  $H^{d_X-1}(X,\mbq) \cong H^{d_X+1}(X,\mbq)$ given by Hard Lefschetz.

By composing the morphisms $\alpha$ and $\beta$ that occur in the sequences \eqref{eq:FirstLongExSeqFinal} and \eqref{eq:SecondLongExSeqFinal}, respectively, we obtain the canonical morphism $\mchp^0 \pi_* j_! {^p}\mbqH_{L^{\vee,*}} \longrightarrow \mchp^0 \pi_* j_* {^p}\mbqH_{L^{\vee,*}}$. This yields the following exact sequence:
\begin{equation}\label{eq:4TermSeq}
0 \longrightarrow i_{0!} \PH^{d_X}(X,\mbq)
\longrightarrow   \mchp^0 \pi_* j_! {^p}\mbqH_{L^{\vee,*}}
\stackrel{\alpha\circ\beta}{\longrightarrow}   \mchp^0 \pi_* j_* {^p}\mbqH_{L^{\vee,*}}
\longrightarrow i_{0!} \KH^{d_X}(X,\mbq)(-1)
\longrightarrow 0,
\end{equation}
where, for any $k\in \dN$,
\[
\PH^{k}(X,\mbq) \coloneqq  \coker(H^{k-2}(X,\dQ)(-1)\rightarrow H^k(X,\dQ))
\]
is the primitive $k$-th cohomology of $X$
\[
\KH^{k}(X,\mbq)(-1) \coloneqq  \ker(H^k(X,\dQ)(-1)\rightarrow H^{k+2}(X,\dQ)).
\]
\begin{cor}\label{cor:WeightIota}
    \begin{enumerate}
        \item The cohomology sheaves $H^\ell\iota_+\cO_Y$ and $H^\ell\iota_\dag \cO_Y$ underlie mixed Hodge modules with weight filtrations of length at most two.
        \item We have
        $$
        W_{d_X+1} H^0 \iota_+ \cO_Y \simeq \overline{j}_{\dag +}\cO_Y
        \qquad\textup{and}\qquad
        \gr^W_{d_X+2} H^0 \iota_+ \cO_Y \simeq i_{0,+} \KH^{d_X}(X,\mbc)(-1)
        $$
        and, dually,
        $$
        W_{d_X} H^0 \iota_\dag \cO_Y \simeq i_{0,+} \PH^{d_X}(X,\mbc)
        \qquad\textup{and}\qquad
        \gr^W_{d_X+1} H^0 \iota_\dag \cO_Y \simeq \iota_{\dag +}\cO_Y
        $$
        In particular, the weight step $W_{d_X} H^0 \iota_\dag \cO_Y \simeq i_{0,+} \PH^{d_X}(X,\mbc)$ resp.\ the graded piece
        $\gr^W_{d_X+2} H^0 \iota_+ \cO_Y \simeq i_{0,+} \KH^{d_X}(X,\mbc)(-1)$ are kernel resp.\ cokernel of the natural morphism
        $$
            H^0\iota_\dag\cO_Y \longrightarrow H^0\iota_+ \cO_Y.
        $$
    \item For $\ell> 0$, $H^\ell\iota_+ \cO_Y$ is pure of weight $d_X+\ell+2$ and we have
    $$
        W_{d_X+\ell+2} H^\ell\iota_+ \cO_Y \simeq i_{0,+} \KH^{d_X+\ell}(X,\mbc)(-1).
    $$
    Dually, for $\ell<0$, $H^\ell\iota_\dag \cO_Y$ is pure of weight $d_X+\ell$ and we have $$
    W_{d_X+\ell} H^\ell\iota_\dag \cO_Y \simeq i_{0,+} \PH^{d_X+\ell}(X,\mbc).
    $$
    \end{enumerate}
\end{cor}
\begin{proof}
\begin{enumerate}
    \item

    In the exact sequence \eqref{eq:FirstLongExSeqFinal}, the Hodge module corresponding to $H^0\iota_\dag \cO_Y$ sits between two objects of weights $d_X+1$ and $d_X+2$. Hence, only these two numbers can occur as
    weight filtration steps on $H^0\iota_\dag \cO_Y$. A similar argument applies to $H^0\iota_+ \cO_Y$, using the exact sequence \eqref{eq:SecondLongExSeqFinal}, where the possible weights are $d_X$ and $d_X+1$.

    \item  The first two isomorphisms follow from the exact sequence \eqref{eq:FirstLongExSeqFinal} and \cref{lem:SuppAtZero} part 2, and \cref{lem:decomp}. The third and fourth isomorphism then follow by duality.

    \item
        From the identification \cref{lem:SuppAtZero}, part 2, together with (the lower rows of) the exact sequence \eqref{eq:FirstLongExSeqFinal},
        we obtain
        $$
        W_{d_X+\ell+1} H^\ell\iota_+ \cO_Y \simeq i_{0,+} \PH^{d_X+\ell+1}(X,\mbc)
        \qquad\textup{and}\qquad
        \Gr^W_{d_X+\ell+2} H^\ell\iota_+ \cO_Y \simeq i_{0,+} \KH^{d_X+\ell}(X,\mbc).
        $$
        However, we have $\PH^{d_X+\ell+1}(X,\mbc)=0$ for $\ell\geq 0$, which gives the purity statement and the weight result.

        The second statement follows by duality (or alternatively by invoking the exact sequence \eqref{eq:SecondLongExSeqFinal}). More precisely, one shows that for $\ell<0$
        $$
        W_{d_X+\ell} H^\ell\iota_\dag \cO_Y \simeq i_{0,+} \PH^{d_X+\ell}(X,\mbc)
        \qquad\textup{and}\qquad
        \Gr^W_{d_X+\ell+1} H^\ell\iota_\dag \cO_Y \simeq i_{0,+} \KH^{d_X+\ell-1}(X,\mbc)
        $$
        and since $\KH^{d_X+\ell-1}(X,\mbc)=0$ for all $\ell\leq 0$, we arrive again at the purity statement we are looking for.
\end{enumerate}
\end{proof}
These statements about the weight filtration on the cohomology sheaves of $\iota_+\cO_Y$ translates into corresponding statements for the local cohomology groups $H^k_{\hat{X}}\cO_V$ by \cref{prop:LocCohomCone}. More precisely, we have the following.

\begin{cornoproof}\label{cor:weightLocCohom}
In the situation of \cref{cor:WeightIota}, the weight filtration on the local cohomology groups of the cone $\hat{X}$ is given by:
\renewcommand{\arraystretch}{1.3}
$$
\Gr_{\ell}^W H^{d_V-d_{\hat{X}}+i}_{\hat{X}}\cO_V
=
\begin{cases}
\iota_{\dag +}\cO_Y    & \text{if } i=0 \text{ and } \ell=2d_V-d_{\hat{X}}, \\
i_{0,+}\KH^{d_X+i}(X,\dC)(-1)   & \textup{if } i \in \{0,\ldots,d_{\hat{X}}-2\} \text{ and } \ell=2d_V-d_{\hat{X}}+i+1,\\
 0   &\text{else.}
\end{cases}
$$
\renewcommand{\arraystretch}{1}
\end{cornoproof}

\section{Colocalization properties of derived tautological systems}
\label{sec:Coloc}

Throughout this and the following section, we consider a connected reductive linear algebraic group $G$ over $\C$ and we denote by $\mathfrak g$ its Lie algebra.
In this section, we discuss the Chevalley--Eilenberg--Spencer--Euler--Koszul complex from \cite{GS-Duality} in more detail and  (see \cref{prop:FunctConstrHatT}) show a colocalization property  extending that for tautological systems from \cite[Theorem 6.3]{GRSSW}.

In \cite{GS-Duality}, we studied for any equivariant $\O_Z$-module $\cM$ on a smooth $G$-variety $Z$ and a character $\beta \colon \mathfrak g \to \C$ the \emph{Chevalley--Eilenberg--Euler--Koszul complex} of $\mathcal D_Z$-modules
  \[
  \cC^\bullet(\cM,\beta) = \mathcal D_Z \otimes_{\O_X} \cM\{\beta\} \otimes_\C \bigwedge^{-\bullet} \mathfrak g
  \]
that represents $\mathcal D_Z \otimes_{\cA_Z}^\mathbb{L} \cM\{\beta\}$, where $\cA_Z \coloneqq \O_Z \otimes \cU(\mathfrak g)$ and
\begin{equation} \label{eq:notationWithBrackets}
\cM\{\beta\} \coloneqq \cM \otimes_{\cO_Z} \cA_Z/\cA_Z(\xi-\beta(\xi) \mid \xi \in \mathfrak g).
\end{equation}
This complex computes the $\mathfrak g$-Lie algebra homology of the $(\mathcal D_Z,\mathfrak g)$-bimodule $\mathcal D_Z \otimes_{\O_Z} \cM \otimes_\C \C_\beta$,  generalizes the Euler--Koszul homology for toric modules form \cite{MMW}, and is a special case of a more general Chevalley--Eilenberg type complex $\cS^\bullet_{\cR|\cU}(\cM)$ in the context of Lie algebroids, see \cite[Definition~2.5]{GS-Duality}.

\begin{lem}\label{lem:DirectImCCompl}
  Let $i \colon Z_1 \hookrightarrow Z_2$ be a $G$-equivariant closed embedding between smooth $G$-varieties. Then
  \[i_+ \cC^\bullet(\cM \otimes_{\O_{Z_1}} \omega_{Z_1}^\vee, \beta) \cong \cC^\bullet(i_* \cM  \otimes_{\O_{Z_2}} \omega_{Z_2}^\vee, \beta)\]
  in $D_{qc}^b(\mathcal D_{Z_2})$ for every $G$-equivariant $\cO_{Z_1}$-module $\cM$. Phrased differently,
  \[i_+ \cC^\bullet(\cM, \beta) \cong \cC^\bullet(i_* (\cM  \otimes_{\O_{Z_1}} \omega_{Z_1|Z_2}), \beta).\]
\end{lem}

\begin{proof}
  From a viewpoint of Lie algebroids, we note that $\cC^\bullet(\cM, \beta) \cong \mathcal D_Z \otimes_{\cA_Z}^\mathbb{L} \cM\{\beta\}$ is the direct image (as discussed in \cite{Che99}) under the morphism of Lie algebroids $(Z, \O_Z \otimes \mathfrak g) \to (Z, \Theta_Z)$ of the left $\cA_Z$-module $\cM\{\beta\}$. The commutative diagram of morphisms of Lie algebroids
   \[\begin{tikzcd}
       (Z_1,\ \O_{Z_1}\otimes \mathfrak g) \ar{r} \ar{d} &
       (Z_2,\ \O_{Z_2} \otimes \mathfrak g) \ar{d} \\
       (Z_1,\ \Theta_{Z_1}) \ar{r} &
       (Z_2,\ \Theta_{Z_2})
     \end{tikzcd}\]
  and the functoriality of the direct image \cite[Proposition~3.1.1]{Che99} imply that $i_+ \cC^\bullet(\cM,\beta) \cong \cC^\bullet(\cN,\beta)$, where $\cN$ is the left $\cA_{Z_1}$-module direct image under the Lie algebroid morphism $(Z_1,\cO_{Z_1} \otimes \mathfrak g) \to (Z_2, \cO_{Z_2} \otimes \mathfrak g)$ induced by the closed embedding $i$. On the level of right modules, this direct image is simply given by $i_*$, as can be checked directly from the definition. On the level of left modules this translates to $\cN \cong i_*(\omega_{Z_1} \otimes_{\cO_{Z_1}} \cM) \otimes_{\cO_{Z_2}} \omega_{Z_2}^\vee \cong i_*(\cM \otimes_{\cO_{Z_1}} \omega_{Z_1} \otimes_{\cO_{Z_1}} i^* \omega_{Z_2}^\vee) = i_*(\cM \otimes_{\cO_{Z_1}} \omega_{X_1|Z_2})$.
                  \end{proof}

\begin{rmk}
  In the proof above, we have used a direct image functor for left modules over Lie algebroids, for which different conventions are possible. In \cite{Che99}, the direct image for right modules over Lie algebroids is defined in a natural way. The passage to a direct image functor for left modules relies on fixing, for each of the Lie algebroids $(X,\Ell)$ involved, one right $\cU(\Ell)$-module $\mathcal E$ that underlies a rank one locally free $\O_X$-module. (Then the functors $\mathcal E \otimes_{\O_X} (\cdot)$ and $\sheafHom_{\O_X}(\cdot,\cE) =(\cdot) \otimes_{\cO_X} \cE^\vee$ give an equivalence of categories between left $\cU(\Ell)$-modules and right $\cU(\Ell)$-modules, see \cite[Theorem~2.2.2]{Che99}.) For such a right module $\cE$, there are two natural intrinsic choices: $\omega_X$ or $\bigwedge^{\rk \Ell} \Ell^\vee$ (which agree for $(X,\Theta_X)$). The statements in the proof above use the convention with $\cE = \omega_X$.
\end{rmk}

As a special case, we recall the following definition from \cite{GS-Duality}:
\begin{dfn}\label{def:HatT}
  Fix a representation $\rho \colon G \to \GL(V)$, the equivariant $\O_V$-module $\overline{i}_*\O_{\overline{Y}}$ for some $G$-orbit closure $\overline{Y} \xhookrightarrow{\overline{i}} V$, and a character $\beta \colon \mathfrak g \to \C$. Then define
  \[
  \hat T(\rho, \overline Y, \beta) \coloneqq \cC^\bullet(\overline{i}_* \cO_{\overline Y},\beta') \quad\text{where}\quad \beta' \coloneqq \trace \circ \differential\rho - \beta,
  \]
viewed as an element of the derived category $D^b_{qc}(\mathcal D_V)$.
\end{dfn}

Then $H^0 \hat T(\rho, \overline Y, \beta)$ is the Fourier-transformed tautological system $\hat \tau(\rho, \overline Y, \beta)$ discussed in \cite{GRSSW}. When restricted to the complement of the boundary orbits, as seen in \cite[displayed formula in proof of Theorem 4.28]{GRSSW}, it can be described as
\[\restr{\hat \tau(\rho, \overline Y, \beta)}{V \setminus \partial Y} = i_+ (\mathcal D_Y \otimes_{\mathcal A_Y} (\omega_Y\{\beta\})^\vee),\]
where, for the open $G$-orbit $Y$ in $\overline Y$, we denote $\partial Y := \overline Y \setminus Y$ and consider $i \colon Y = \overline Y \setminus \partial Y \hookrightarrow V \setminus \partial Y$, the closed embedding of $Y$ into $V \setminus \partial Y$. The derived statement remains true:

\begin{prop}\label{prop:ThatRestricted}
  Let $\partial Y \coloneqq \overline Y \setminus Y$ and $i \colon Y \hookrightarrow V \setminus \partial Y$. Then
  \[\restr{\hat T(\rho,\overline{Y},\beta)}{V \setminus \partial Y} = i_+ (\mathcal D_Y \otimes^\mathbb{L}_{\mathcal A_Y} (\omega_Y\{\beta\})^\vee).\]
\end{prop}

\begin{proof}
  Note that $\restr{\hat T(\rho, \overline Y, \beta)}{V \setminus \partial Y} = \cC^\bullet(i_*\O_Y, \beta')$.
    Then \cref{lem:DirectImCCompl}
  (applying its first formula to the module $\cM = \O_Y$ and parameter $-\beta$) shows
  that
  \[i_+ \cC^\bullet(\omega_Y^\vee,-\beta) \cong \cC^\bullet(i_*\O_Y \otimes \omega_{V\setminus \partial Y}^\vee, -\beta) \cong \cC^\bullet(i_*\O_Y, \trace \circ \differential \rho -\beta) = \cC^\bullet(i_*\O_Y, \beta'),\]
  where the middle isomorphism is due to $\omega_{V \setminus \partial Y}^\vee \cong \O_{V \setminus \partial Y} \{\trace \circ \differential\rho\}$ by \cite[Lemma~4.4]{GS-Duality}.
        Since $\cC^\bullet(\omega_Y^\vee,-\beta) = \cD_Y \otimes_{\cA_Y}^\mathbb{L} (\omega_Y\{\beta\})^\vee$ by definition, this shows the claim.
\end{proof}

\cref{prop:ThatRestricted} describes $\hat T(\rho, \overline Y, \beta)$ outside the boundary orbits $\partial Y$ of the orbit closure $\overline Y$ (and we give a more detailed description of $\cD_Y \otimes_{\cA_Y}^\mathbb{L} (\omega_Y\{\beta\})^\vee$ below in \cref{sec:Orbit}). Understanding the behavior along the boundary orbits is in general a subtle problem (see for example \cite{SW09} in the context of toric varieties and GKZ systems or \cite{NarvaezSevenheck} for certain free divisors). A particular question for given $G$-orbit closures $\overline Y \subseteq V$ is to understand when
\[
 \hat T = \hat T(\rho, \overline Y, \beta) \text{ is } \begin{cases} \text{localized in the sense that } &\hat T = j_+ (\restr{\hat T}{V \setminus \partial Y}), \\
\text{colocalized in the sense that } &\hat T = j_\dag (\restr{\hat T}{V \setminus \partial Y})
\end{cases}
\]
for $j \colon V \setminus \partial Y \hookrightarrow V$.
We now discuss this (co-)localization property for a special class of the complexes $\hat{T}$ that arise from affine cones over homogeneous spaces.
\begin{nota}
For the remainder of this section, we consider
a projective homogeneous space $X = G_0/P_0$, a very ample $G_0$-equivariant line bundle $\Ell$ on $X$ and the embedding $X \hookrightarrow \P V$ induced by the complete linear system $|\Ell|$. Furthermore, $V$ will be dual to the corresponding vector space of sections, $\hat X$ the affine cone over $X$ in $V$.
Finally, set
$G \coloneqq  \C^* \times G_0$ with associated Lie algebras $\mathfrak g =\C \mathbf e \oplus \mathfrak g_0$. Note that $Y := \hat X \setminus \{0\}$ is a $G$-orbit in $V$ and its orbit closure is $\overline Y =\hat X$.
\end{nota}

The first result in this context is a preliminary vanishing statement needed in order to show \cref{prop:FunctConstrHatT} below. It is an upgraded version of the direct computations done in \cite[Proof of Proposition~6.7]{GRSSW}. Here we use Lie theoretic methods (and we thank Avi Steiner for communicating the proof idea to us).
\begin{prop}\label{prop:Acyclic}
  Assume that $\dim X > 0$; then                 \[R\Gamma(V,\omega_V \otimes^\mathbb{L}_{\mathcal A_V} \O_{\hat X}\{\beta'\}) \cong
  \begin{cases}
  0 &\text{if } \beta \neq 0, \\
  H_{-\bullet}(G,\C) &\text{if } \beta = 0,
  \end{cases}\]
  using the notation \eqref{eq:notationWithBrackets} and the definition of $\beta'$ from \cref{def:HatT}.
\end{prop}
Before entering the proof, let us recall the following well-known results from Lie theory that we will also use below in \cref{sec:Orbit}.

\begin{lem}\label{lem:VanishLieHomReductive}
Let $\fg$ be a reductive Lie algebra, and let $V$ be a simple non-trivial $\fg$-module. Then $H_k(\fg,V)=0$ for all $k$.
\end{lem}

\begin{proof}
  If $\mathfrak g$ is semisimple, this classical vanishing result can be found for example in \cite[Th.~7.8.9]{Weibel}. In general, since $\mathfrak g$ is reductive, we have $\mathfrak g \cong \mathfrak a \oplus \mathfrak g^{ss}$, where $\mathfrak a$ is abelian
    and $\mathfrak g^{ss}$ is semisimple.
    This gives
  \begin{equation} \label{eq:splittingReductiveLieHomology}
  H_k(\mathfrak g, V) \cong \bigoplus_{i+j=k} H_i(\fa,V) \otimes_\C H_j(\mathfrak g^{ss},V).
  \end{equation}
  Since $[\mathfrak a, \mathfrak g] = 0$ and since $V$ is simple, Schur's lemma implies that the action of $\mathfrak a$ is scalar multiplication on $V$ through some character $\chi \colon \mathfrak a \to \C$. (The action of each element of $\mathfrak a$ on $V$ is a $\mathfrak g$-module endomorphism which must by Schur's lemma be multiplication with some scalar.) We distinguish two cases.

  If $\chi \neq 0$,  pick $a \in \mathfrak a$ with $\chi(a) \neq 0$. Writing $\mathfrak a = \C a \oplus \mathfrak b$,
  \[
  H_k(\mathfrak a, V) \cong \bigoplus_{i+j=k} H_i(\C a,V) \otimes_\C H_j(\mathfrak b,V).
  \]
  But $H_\bullet(\C a,V)=H_\bullet(\dots \to 0 \to V \xrightarrow{\chi(a) \cdot } V \to 0 \to \dots)=0$, since $\chi(a) \in \C^*$, and we conclude $H_\bullet(\CC a,V)=0$ in this case.

  If $\chi = 0$, then $\mathfrak a$ acts trivially on $V$. Then $V$ being a simple non-trivial $\mathfrak g$-module means that $V$ is a simple non-trivial $\mathfrak g^{ss}$-module. Hence, $H_\bullet(\mathfrak g^{ss},V) = 0$ by \cite[Th.~7.8.9]{Weibel}.

  Thus, in both cases, at least one of  $H_\bullet(\mathfrak a, V)$ and $H_\bullet(\mathfrak g^{ss},V)$ vanishes. By \cref{eq:splittingReductiveLieHomology}, $H_\bullet(\mathfrak g, V) = 0$ in all cases.
\end{proof}

\begin{lem} \label{lem:VanishingLieHomReductiveGroup}
Let $V$ be a finite-dimensional representation of a connected reductive algebraic group $G$. Then $H_k(\fg,V) \cong H_k(G,\C) \otimes_\C V^G$ for all $k$. In particular, $H_k(\fg,V)=0$ for all $k$ if $V$ has no $G$-invariants.

\end{lem}

\begin{proof}
Reductive groups over $\C$ are linearly reductive (i.e., their finite-dimensional representations decompose as direct sums of irreducible representations, see \cite[Corollary~22.43]{Milne}), and so we may write $V =\bigoplus_{i=1}^m V_i$ with $V_i$ irreducible. We may assume $V_i$ is non-trivial for $i \leq r$ and $V_i = \C$ for $i>r$, then $V = V^G \oplus \bigoplus_{i=1}^r V_i$. Irreducible $G$-representations are simple $\mathfrak g$-modules (this can be checked by viewing $G$ as a connected complex Lie group, and then it follows using the exponential map, see, e.g., \cite[Corollary~3.47]{Hall}). In particular, $H_k(\fg, V_i) = 0$ for $i \leq r$ by \cref{lem:VanishLieHomReductive}. Hence,
\[H_k(\fg, V) \cong H_k(\fg, V^G) \oplus \bigoplus_{i=1}^r H_k(\fg, V_i) = H_k(\fg, V^G) \cong H_k(\fg, \C) \otimes_\C V^G.\]

Finally, for the trivial representation $\C$, Lie algebra homology computes the homology of $G$, i.e.\ $H_k(\mathfrak g, \C) \cong H_k(G,\C)$, which can be seen as follows: Since $G$ is reductive, there is a compact real Lie group $K$ with $\mathfrak g = \mathfrak k \otimes_\R \C$ such that $G$ is diffeomorphic to $K \times \mathfrak k$ with $\mathfrak k = \operatorname{Lie}(K)$ considered as an $\R$-vector space (see e.g.\ \cite[Theorem~1.9]{SalamonLieGroups} or
\cite[Chapter~5, §2, Theorem~8]{Onishchik-Vinberg}). For compact real Lie groups, the cohomology $H^i(K,\R)$ can be computed as Lie algebra cohomology  $H^\bullet(K,\R) \cong H^\bullet(\mathfrak k, \R)$ by the complex $\bigwedge^\bullet \mathfrak k^\vee$, this is a very classical result by Chevalley--Eilenberg, see \cite[Theorem~15.2]{ChevalleyEilenberg}. After dualizing, this implies $H_i(K,\R) \cong H_i(\mathfrak k,\R)$ and so
  \[
  H_\bullet(G,\C) \cong H_\bullet(K \times \R^k,\R) \otimes_\R \C \cong H_\bullet(K,\R) \otimes_\R \C \cong H_\bullet(\mathfrak k, \R) \otimes_R \C \cong H_\bullet(\mathfrak k \otimes_\R \C,\C) \cong H_\bullet(\mathfrak g, \C). \qedhere
  \]
\end{proof}

\begin{proof}[Proof of \cref{prop:Acyclic}]
The object $\omega_V \otimes^\mathbb{L}_{\mathcal A_V} \O_{\hat X}\{\beta'\}$ in $D^b(\C_V)$ is represented by the complex
\[\omega_V \otimes_{\O_V} \O_{\hat X}\{\beta'\} \otimes_{\O_V} \bigwedge_{\O_V}^{-\bullet} (\O_V \otimes \mathfrak g) \cong \omega_V \otimes_{\O_V} \O_{\hat X}\{\beta'\} \otimes_{\C} \bigwedge^{-\bullet} \mathfrak g,\]
denoted $\mathcal S^\bullet_{\C|\mathcal A_V}(\omega_V \otimes_{\O_V} \O_{\hat X}\{\beta'\})$ in \cite{GS-Duality} (see specifically \cite[Definition 2.5]{GS-Duality}). This follows from \cite[Lemma~2.10 and 2.12]{GS-Duality}, see also \cite[Lemma 6.12]{GRSSW}.
Notice in particular that the terms of $\mathcal S^\bullet_{\C|\mathcal A_V}(\omega_V \otimes_{\O_V} \O_{\hat X}\{\beta'\})$ are $\cO_V$-coherent (though the differentials are not $\cO_V$-linear) and therefore $\Gamma(V,\cdot)$-acyclic, hence
$$
R\Gamma(V,\omega_V \otimes^\mathbb{L}_{\mathcal A_V} \O_{\hat X}\{\beta'\})
= \Gamma(V,\mathcal S^\bullet_{\C|\mathcal A_V}(\omega_V \otimes_{\O_V} \O_{\hat X}\{\beta'\})).
$$
The cohomology of the complex $\Gamma(V,\mathcal S^\bullet_{\C|\mathcal A_V}(\omega_V \otimes_{\O_V} \O_{\hat X}\{\beta'\}))$ is the Lie algebra homology of the right $\mathfrak g$-module $\Gamma(V,\omega_V \otimes_{\O_V} \O_{\hat X}\{\beta'\})$, see \cite[Remark 2.8]{GS-Duality}.

From the definitions of the $\mathfrak g$-action, one verifies that the corresponding left $\mathfrak g$-module
is
\[
\Gamma(V,(\omega_V \otimes_{\O_V} \O_{\hat X}\{\beta'\}))^\text{left} \cong \Gamma(V,\O_{\hat X}\{\beta'-\trace \circ \differential\rho\}) = \Gamma(V,\O_{\hat X}\{-\beta\}) \cong S \otimes \C_{-\beta},
\]
where
$S \coloneqq \Gamma(V,\O_{\hat X})$ is the homogeneous coordinate ring of $X$ embedded into $\P V$ by the complete linear system $|\Ell|$ and
\[\C_{-\beta}\coloneqq \mathcal{U}(\mathfrak{g})/\mathcal U(\mathfrak g)(\xi+\beta(\xi)\mid \xi\in \mathfrak{g})\]
is the one-dimensional $\mathfrak g$-module corresponding to the character $-\beta$.
Hence,
\begin{equation} \label{eq:reductionToLieHomology}
R\Gamma(V,\omega_V \otimes^\mathbb{L}_{\mathcal A_V} \O_{\hat X}\{\beta'\}) \cong H_{-\bullet}(\mathfrak g, S \otimes \C_{-\beta}),
\end{equation}
and we wish to show the vanishing of this Lie algebra homology.
Decomposing $S$ into its graded pieces $S = \bigoplus_{d \in \N} S_d$, we get a decomposition of left $\mathfrak g$-modules
\[S \otimes \C_{-\beta} = \bigoplus_{d \in \N} S_d \otimes \C_{-\beta}.\]

First, consider the case $d > 0$. We have $S \subseteq \bigoplus_{d \in \N} \Gamma(X,\Ell^{\otimes d})$ as $\mathfrak g$-modules, so  $S_d$ is a $\mathfrak{g}$-submodule of $\Gamma(X,\Ell^{\otimes d})$. Since $\Ell$ is an equivariant line bundle on a projective homogeneous space $X$, the theorem of Borel--Weil (see, e.g., \cite[Cor.~to Th.~V]{Bot57}) implies that $\Gamma(X,\Ell^{\otimes d})$ is a simple $\mathfrak{g}$-module. In particular, $S_d$ must in fact agree with $\Gamma(X,\Ell^{\otimes d})$---whence  the embedding of $X$ into $\P V$ is projectively normal---and must be a simple $\mathfrak{g}$-module. As $\dim \C_{-\beta} = 1$,   $S_d \otimes_\C \C_{-\beta}$ is also  a simple $\mathfrak g$-module. From $\dim X > 0$ and $d > 0$, we can conclude that $\dim_\C (S_d \otimes_\C \C_{-\beta}) = \dim_\C S_d > 1$, so this is not the trivial $\mathfrak g$-module. By \cref{lem:VanishLieHomReductive}, this shows
\begin{equation} \label{eq:positiveDegVanishing}
H_\bullet(\mathfrak g, S_d \otimes \C_{-\beta}) \neq 0.
\end{equation}

For $d = 0$, note that $S_0 \otimes \C_{-\beta} = \C_{-\beta}$ is a one-dimensional (hence simple) $\mathfrak g$-module, on which $\mathfrak g$ acts through $\beta$. By \cref{lem:VanishLieHomReductive}, this shows vanishing of its Lie algebra homology for $\beta \neq 0$, while in the case $\beta = 0$, we deal with the trivial representation, so Lie algebra homology computes singular homology $H_\bullet(G, \C)$ by \cref{lem:VanishingLieHomReductiveGroup}. From this we now conclude
\[
R\Gamma(V,\omega_V \otimes^\mathbb{L}_{\mathcal A_V} \O_{\hat X}\{\beta'\}) \cong H_{-\bullet}(\mathfrak g, S \otimes \C_{-\beta}) = H_{-\bullet}(\mathfrak g, S_0 \otimes \C_{-\beta}) =
\begin{cases}
  H_{-\bullet}(G, \C) &\text{if } \beta = 0, \\
  0 &\text{if }\beta \neq 0 \hspace{2em}
\end{cases}
\]
via
Equations~\eqref{eq:reductionToLieHomology} and \eqref{eq:positiveDegVanishing}.
\end{proof}

Using
\cref{prop:Acyclic}, we obtain the following generalization of \cite[Theorem 6.3]{GRSSW}:
\begin{prop} \label{prop:FunctConstrHatT}
  Assume $\dim X > 0$. For $\beta \neq 0$,
  $\hat T(\rho, \hat X, \beta)$ is colocalized in the sense that
  \[\hat T(\rho, \hat X, \beta) \cong j_\dag j^+ \hat T(\rho, \hat X, \beta)\]
  in $D^b_{qc}(\mathcal{D}_V)$.
  More precisely, for $\beta \neq 0$, we have an isomorphism   \[\hat T(\rho,\hat X, \beta) \cong \iota_\dag (\mathcal D_{\hat X \setminus \{0\}} \otimes^{\mathbb{L}}_{\mathcal A_{\hat X \setminus \{0\}}} (\omega_{\hat X \setminus \{0\}}\{\beta\})^\vee),\]
  where $\iota \colon \hat X \setminus \{0\} \hookrightarrow V$.

  If $\hat X$ is Gorenstein, then for $\beta = 0$, the localization property
  \[\hat T(\rho, \hat X, 0) \cong j_+ j^+ \hat T(\rho, \hat X, 0)\]
  holds and more precisely, we have
  \[
     \hat T(\rho, \hat X, 0) \cong \iota_+ (\mathcal D_{\hat X \setminus \{0\}} \otimes^{\mathbb{L}}_{\mathcal A_{\hat X \setminus \{0\}}} \omega_{\hat X \setminus \{0\}}^\vee).    \]
                          \end{prop}

\begin{proof}
First, consider the case $\beta \neq 0$. We have the adjunction triangle (compare to \cite[Formula (21)]{GRSSW})
$$
j_\dag j^+ \hat T(\rho,\hat X, \beta) \longrightarrow
\hat T(\rho,\hat X, \beta) \longrightarrow
\left(i_{0,+} i_0^\dag \hat T(\rho,\hat X, \beta) [d_V]\right) \stackrel{+1}{\longrightarrow},
$$
where
$$
\begin{tikzcd}
\hat{X} \setminus \{0\} \ar[hook]{r}{i} & V\setminus\{0\} \ar[hook]{r}{j} & V & \{0\} \ar[hook', swap]{l}{i_0}.
\end{tikzcd}
$$

For the colocalization property, we therefore need to show that the complex of $\dC$-vector spaces $i_0^\dag \hat T(\rho,\hat X, \beta)$ is acyclic. We have the isomorphism
$$i_0^\dag \hat T(\rho,\hat X, \beta)[d_V] \cong
a_{V,+} \hat T(\rho,\hat X, \beta)
$$
in $D^b(\dC)$ by invoking \cite[Lemma 3.3]{ReichWalth-Duco} (where the result is shown only for $\cD_V$-modules, but inspection of the proof shows that it holds for complexes as well; see also \cite[Lemma 4.4]{AviDualProjRestrGKZ} where the case of complexes and the more general twisted $\dC^*$-equivariant case are discussed).
Now,
\begin{align*}
a_{V,+} \hat T(\rho,\hat X, \beta)
&= Ra_{V,*} \left(\omega_V \otimes_{\mathcal{D}_V}^\dL  \hat T(\rho,\hat X, \beta)\right)
\cong Ra_{V,*}\left(\omega_V \otimes_{\mathcal{D}_V}^\dL  \left(\mathcal D_V \otimes^\mathbb{L}_{\mathcal A_V} \O_{\hat X}\{\beta'\}\right)\right) \\
&\cong
Ra_{V,*}\left(\omega_V \otimes^\mathbb{L}_{\mathcal A_V} \O_{\hat X}\{\beta'\}\right).
\end{align*}
By \cref{prop:Acyclic},
this is zero provided that $\beta \neq 0$. This proves the colocalization property.

By \cref{prop:ThatRestricted}, we know
$$
j^+ \hat T(\rho,\hat X, \beta) \cong i_+(\mathcal D_{\hat X \setminus \{0\}} \otimes^{\mathbb{L}}_{\mathcal A_{\hat X \setminus \{0\}}} (\omega_{\hat X \setminus \{0\}}\{\beta\})^\vee).
$$
This leads to
\begin{equation}
\label{eq:colocalizationInProof}
\hat T(\rho, \hat X, \beta) \cong j_\dag j^+ \hat T(\rho, \hat X, \beta) \cong \iota_\dag(\mathcal D_{\hat X \setminus \{0\}} \otimes^{\mathbb{L}}_{\mathcal A_{\hat X \setminus \{0\}}} (\omega_{\hat X \setminus \{0\}}\{\beta\})^\vee)
\end{equation}
using $\iota = j \circ i$.

In the Gorenstein case, we may apply the duality result from \cite[Theorem~4.6]{GS-Duality}, stating that
\[
\mathbb D \hat T(\rho, \hat X, 0) \cong \hat T(\rho, \hat X, \gamma)[d_{\hat X}-d_G]
\]
for some non-zero $\gamma \colon \fg \to \C$. Then the colocalization property $j_\dag j^+ \hat T(\rho, \hat X, \gamma) \cong \hat T(\rho, \hat X, \gamma)$ just shown
in \eqref{eq:colocalizationInProof}
implies by duality the localization property $j_+ j^+ \hat T(\rho, \hat X, 0) \cong \hat T(\rho, \hat X, 0)$, which together with \cref{prop:ThatRestricted} leads to the claimed description of $\hat T(\rho, \hat X, 0)$.
\end{proof}

\begin{rmk}\label{rem:GorensteinGamma}
  As observed in \cite[Example~4.8]{GS-Duality}, $\hat X$ is Gorenstein if and only if, for some $\ell \in \Z_{>0}$, we have $\Ell^{\otimes \ell} \cong \omega_X^\vee$ as line bundles on $X$. Then $(\Ell^{\otimes \ell} \otimes \omega_X)^\vee$ is a trivial line bundle with an equivariant structure. Hence, $G$ acts on sections of this line bundle by scalar multiplication through a character $\mu \colon G \to \C^*$. Note that the restriction of $\mu$ to the subgroup $\C^* \subseteq G$ is given by $\ell$-th powers. For $\gamma \coloneqq \differential \mu \colon \fg \to \C$, we have $\omega_{\hat X \setminus \{0\}} \cong \O_{\hat X \setminus \{0\}}\{-\gamma\}$ and therefore, according to \cref{prop:FunctConstrHatT}:
  \[\hat T(\rho, \hat X, \gamma) \cong \iota_\dag(\cD_{\hat X \setminus \{0\}} \otimes_{\cA_{\hat X \setminus \{0\}}}^\mathbb{L} \cO_{\hat X \setminus \{0\}}).\]
  According to the discussion in \cite[§4.4]{GRSSW} (there for case $\restr{\gamma}{\mathfrak g_0} = 0$ only, but the same argument applies in general), the parameters $\beta = \gamma$ and $\beta = 0$ are the only two parameters for which $H^0 \hat T(\rho, \hat X, \beta) \neq 0$ and we confirm that these are in fact the only integrable parameters for which $\hat T(\rho, \hat X, \beta) \neq 0$ in \cref{rem:NonZeroTHat} below.

In general, non-vanishing of $H^0 \hat T (\rho, \hat X, \beta)$ for some $\beta \neq 0$ is known to occur only if $\hat X$ is $\Q$-Gorenstein and the case that such $\beta \colon \mathfrak g \to \C$ integrates to a character of $G$ is precisely the Gorenstein case just discussed, see \cite[Example~4.8]{GS-Duality}.
\end{rmk}

\begin{cor}\label{cor:DualityMorphism}
Let $\dim X > 0$ and assume that $\hat X$ is Gorenstein and let $\gamma \colon \fg \to \C$ be such that $\omega_{\hat X\setminus \{0\}} \cong \O_{\hat X \setminus \{0\}}\{-\gamma\}$. Then there is a morphism
$$
\hat T(\rho,\hat X, \gamma) \longrightarrow
\hat T(\rho,\hat X, 0)[d_{\hat X}-d_G]
$$
in $D_{qc}^b(\mathcal{D}_V)$.
\end{cor}

\begin{proof}
By \cite[Proposition 4.3]{GS-Duality} we know that $\hat T(\rho,\hat X, \beta)$ has cohomologies in degrees $\{d_{\hat{X}}-d_G,\ldots,0\}$ only for any $\beta$. Hence
a morphism
$\hat T(\rho,\hat X, \gamma) \to
\hat T(\rho,\hat X, 0)[d_{\hat X}-d_G]$
is given by the corresponding morphism on the zero-th cohomology (as can be seen by representing these elements in the derived categories by complexes in degrees $\{d_{\hat X}-d_G,\ldots,0\}$ and $\{0,\ldots,d_G-d_{\hat X}\}$, respectively).
From \cref{prop:FunctConstrHatT}, we know that $\hat T(\rho, \hat X, \gamma) = \iota_\dag (\cD_{\hat X \setminus \{0\}} \otimes_{\cA_{\hat X \setminus \{0\}}}^\mathbb{L} \cO_{\hat X \setminus \{0\}})$. It then follows from the third quadrant spectral sequence
\[E_2^{pq} = H^p\iota_\dag H^q (\cD_{\hat X \setminus \{0\}} \otimes_{\cA_{\hat X \setminus \{0\}}}^\mathbb{L} \cO_{\hat X \setminus \{0\}}) \Longrightarrow H^{p+q}\hat T(\rho, \hat X, \gamma),\]
that
\[
H^0 \hat T(\rho, \hat X, \gamma) \cong
H^0 \iota_\dag H^0 (\cD_{\hat X \setminus \{0\}} \otimes_{\cA_{\hat X \setminus \{0\}}}^\mathbb{L} \O_{\hat X \setminus \{0\}})
\cong H^0 \iota_\dag \O_{\hat X \setminus \{0\}},
\]
where the last isomorphism uses $\cD_{\hat X \setminus \{0\}} \otimes_{\cA_{\hat X \setminus \{0\}}} \O_{\hat X \setminus \{0\}} \cong \cD_{\hat X \setminus \{0\}}/\cD_{\hat X \setminus \{0\}} \mathfrak g \cong \O_{\hat X \setminus \{0\}}
$ by transitivity of the action of $G$ on $\hat X \setminus \{0\}$. (Notice that this is a description of a Fourier-transformed tautological system $\hat\tau(\rho, \hat X, \gamma)$ as $H^0\iota_\dag \O_{\hat X \setminus \{0\}}$; compare also \cite[Theorem~6.3]{GRSSW}.)

Taking duals, we obtain
\[H^0 \iota_+ \O_{\hat X \setminus \{0\}} \cong \mathbb{D} H^0\hat T(\rho, \hat X, \gamma) \cong H^0 \hat T(\rho, \hat X, 0)[d_{\hat X}-d_G],\]
using the duality result proved in \cite[Theorem~4.6]{GS-Duality}.
The natural morphism $H^0\iota_\dag \O_{\hat X \setminus \{0\}} \to H^0 \iota_+ \O_{\hat X \setminus \{0\}}$ therefore gives a morphism
\[H^0 \hat T(\rho, \hat X, \gamma) \to H^0 \hat T(\rho, \hat X, 0)[d_{\hat X}-d_G].\]
This determines the desired morphism $\hat T(\rho,\hat X, \gamma) \to \hat T(\rho,\hat X, 0)[d_{\hat X}-d_G]$ by the discussion of the cohomological amplitude above.
\end{proof}

As we have just seen, the morphism of complexes $\hat T(\rho,\hat X, \gamma) \longrightarrow
\hat T(\rho,\hat X, 0)[d_{\hat X}-d_G]$ is determined by its degree zero part.
Using our description of the weight filtration from \cref{cor:WeightIota}, part 2., we obtain the following description of the kernel and cokernel of the degree zero part of this morphism.

\begin{cor}\label{cor:ImKerDualMorphism}
Under the assumptions of \cref{cor:DualityMorphism}, the image of
$$
H^0\hat T(\rho,\hat X, \gamma) \longrightarrow
H^0\hat T(\rho,\hat X, 0)[d_{\hat X}-d_G]
$$
equals $H^0\iota_{\dag +} \O_{\hat X \setminus \{0\}}$. Moreover, we have
$$
\ker \left(H^0 \hat T(\rho,\hat X, \gamma) \longrightarrow
H^{0}\hat T(\rho,\hat X, 0)[d_{\hat X}-d_G] \right) \cong i_{0,+} \PH^{d_X}(X,\dC)
$$
and
$$
\coker \left(H^0 \hat T(\rho,\hat X, \gamma) \longrightarrow
H^0\hat T(\rho,\hat X, 0)[d_{\hat X}-d_G] \right)= i_{0,+} \KH^{d_X}(X,\dC).
$$
\end{cor}

\begin{rem}
We have excluded the case $\dim(X)=0$ from the previous discussion since
in that case \cref{prop:Acyclic} can fail, and also because this is of no interest in any geometrically relevant situation. If indeed $\dim(X)=0$ (so that $\hat{X}= V = \dA^1$), an argument  using \cref{prop:HatTOrbit} below can me made that  replaces \cref{prop:Acyclic} and shows that
then
$$
\hat{T}(\rho, \hat X, \beta)=
\begin{cases}
\D j_+\cO_{\dC^*}\otimes H_{-\bullet}(G_0,\dC) &\text{if }\beta\in \dZ_{\leq 0},
\\
\D j_\dag\cO_{\dC^*}\otimes H_{-\bullet}(G_0,\dC) &\text{if }\beta\in \dZ_{> 0},
\end{cases}
$$
where $j:\hat{X}\backslash\{0\}=\dC^*\hookrightarrow \hat{X}=\dC$.
\end{rem}

\section{Restriction to the open orbit}\label{sec:Orbit}

In this section, we give a topological interpretation of the restriction of the complex $\hat{T}(\rho,\overline{Y},\beta)$ to a stratum.
\cref{prop:ThatRestricted} (and \cref{prop:FunctConstrHatT} in the case of homogeneous spaces) describes $\hat T(\rho, \overline Y, \beta)$ in terms of the $\mathcal D_Y$-module $\mathcal D_Y \otimes_{\mathcal A_Y}^{\mathbb L} (\omega_Y\{\beta\})^\vee$. Our aim here is  to understand this object on $Y$. As a consequence of these considerations we give, via \cref{prop:LocCohomCone}, an interpretation of this complex in terms of local cohomology modules.

Throughout this section, $G$ denotes a connected reductive linear algebraic group acting transitively on an algebraic variety $Y$.
We here restrict ourselves to the case that $\beta$ is integrable, i.e., $\beta = \differential \nu$ for some group character $\nu \colon G \to \C^*$. Under this assumption, $\omega_Y\{\beta\}$ underlies a $G$-equivariant line bundle. Notice that in the case that the reductive group $G$ is of the form $G = (\C^*)^r \times G_0$ with $G_0$ semisimple, then $\restr{\beta}{\mathfrak g_0} = 0$ and the integrability condition of $\beta \colon \bigoplus_{i=1}^r \C e_i \oplus \fg_0 \to \C$ is equivalent to the integrality of the parameters: $\beta(\mathbf e_i) \in \Z$ for all $i$. An arbitrary reductive group $G$ admits a surjection $(\C^*)^r \times G_0 \twoheadrightarrow G$ of algebraic groups with finite kernel, where $G_0$ is semisimple. Then $G$ and $(\C^*)^r \times G_0$ share the same Lie algebra $\C^r \oplus \mathfrak g_0$ and every group character $\nu \colon G\to \C^*$ induces a group character $(\C^*)^r \oplus G_0 \to \C^*$ and therefore an integrable $\beta \colon \mathfrak g = \mathfrak g_0 \oplus \bigoplus_{i = 1}^r \C \mathbf e_i \to \C$ necessarily satisfies integrality $\beta(\mathbf e_i) \in \Z$ for all $i$ (but this may no longer be a sufficient condition: integrability is equivalent to $(\beta(\mathbf e_i))_{i=1}^r$ lying in a certain sublattice of $\Z^r$).

We first discuss a lemma that holds in greater generality and uses equivariance properties of the objects involved.

\begin{lem}\label{lem:RestToPoint}
    Fix $y_0\in Y$ and assume that $P\coloneqq\textup{Stab}(y_0) \subseteq G$ is connected. Denote by
  $\pi \colon G\rightarrow Y=G/P$, $g \mapsto g \cdot y_0$ the canonical projection. Then we have an isomorphisms in
in $D^b_h(\mathcal{D}_Y)$:
  $$
    \pi_\star \cO_G \cong  \cO_Y\otimes_\dC a_\star \cO_P,
  $$
  for $\star\in\{+,\dag\}$,
  where, as before, $a:P\to\{y_0\}$ is the structure morphism of $P$.
  \end{lem}
\begin{proof}

Extending the usual definition of equivariance, we call an object $\cM\in D^b_{qc}(\mathcal{D}_Y)$ (strongly) $G$-equivariant if there is an isomorphism
in $D^b_{qc}(\cD_{G\times Y})$
\begin{eqnarray}\label{eqn-str-equiv}
\act^+\cM \cong \textup{pr}_2^+\cM,
\end{eqnarray}
(satisfying a cocycle condition), where $\act:G\times Y \to Y$ is the given $G$-action. If no confusion is possible, we will usually speak about $G$-equivariant objects only, which always refers to strong $G$-equivariance.

It is easy to see that if $Y_1$ and $Y_2$ are smooth $G$-varieties and if $f:Y_1\rightarrow Y_2$ is $G$-equivariant, then for an equivariant object $\cM\in D^b_{qc}(\mathcal{D}_{Y_1})$ resp.\ $\cN\in D^b_{qc}(\mathcal{D}_{Y_2})$, the direct resp.\ inverse image $f_+\cM$ resp.\ $f^+\cN$ as well as their counterparts $f_\dag\cM$ resp.\ $f^\dag\cN$ are equivariant as well. In particular, in the situation of the lemma, the complexes $\pi_\star \cO_G$ for $\star\in\{+,\dag\}$ are equivariant objects in $D^b_{qc}(\cD_Y)$.

With these notations at hand, and assuming that $G$ acts transitively on $Y$, one can easily show by generalizing \cite[Proof of Theorem 11.6.1 (i)]{Hotta} that for an equivariant object $\cM\in D^b_{qc}(\mathcal{D}_Y)$ we have
\begin{equation}\label{eq:PullbackHottaProof}
\pi^+ \cM \cong \cO_G \otimes_{\dC} i^+_0 \cM,
\end{equation}
where $i_0:\{y_0\}\hookrightarrow Y$.

It is well-known that if the stabilizer $P$ is connected, then the inverse image functor $\pi^+$ induces an equivalence between
$D^b_{qc}(\mathcal{D}_Y)$ and the derived category $D^b_{qc}(\textup{Mod}(\cD_G)^P)$ of $P$-equivariant $\cD_G$-modules (notice that being an object in the latter is a priori a stronger condition than being strongly $P$-equivariant in the sense of \eqref{eqn-str-equiv}). A non-derived version of this statement is in \cite[Proposition 8.4]{Gaitsgory}; remark that the necessary conditions for this conclusion are much weaker:  we only need that $\pi$ is a principal $P$-fiber bundle.
Moreover, the forgetful functor $\textup{Mod}(\mathcal{D}_G)^P$ to $\textup{Mod}(\mathcal{D}_G)$ is fully faithful by \cite[Proposition 3.1.2]{vanDenBergh}. Using the corresponding statement for the derived categories, we conclude that if $\cM,\cN\in D^b_{qc}(\mathcal{D}_Y)$ satisfy $\pi^+\cM\cong \pi^+\cN$, then $\cM\cong\cN$. By \cref{eq:PullbackHottaProof}, an equivariant object $\cM\in D^b_{qc}(\mathcal{D}_Y)$ satisfies
$\pi^+(\cM)\cong \pi^+(\cO_Y\otimes_\dC i_0^+\cM)$, and hence
\begin{equation}\label{eq:EquiModTrivial}
\cM \cong  \cO_Y\otimes_\dC i_0^+\cM
\end{equation}
in $D^b_{qc}(\cD_Y)$. In particular, any such $\cM$ is necessarily a smooth $\cD_Y$-module, so that $i_0^+\cM\cong i_0^\dag \cM$. Applying base change and  \cref{eq:EquiModTrivial} to the case $\cM=\pi_\star\cO_G$, we obtain the desired result.
\end{proof}

\begin{rmk} \label{rmk:Formality}
  It follows from \cref{eq:EquiModTrivial} that an $G$-equivariant object $\cM \in D^b_{qc}(\cD_Y)$ for transitive actions with connected stabilizer is \emph{formal}: it is quasi-isomorphic to the complex built from its cohomology groups $H^k \cM[-k]$ with trivial differential.  In particular,  $G$-equivariant objects in $D_{qc}^b(\cD_Y)$ are objects in the category $D^b_{qc}(\textup{Mod}(\cD_Y)^G)$ and two such objects are isomorphic if and only if they have isomorphic cohomology.
\end{rmk}

\begin{prop}\label{prop:HatTOrbit}
  Assume $G$ acts transitively on $Y$ such that the stabilizer $P$ of a point $y_0 \in Y$ is connected. Then
  \[\mathcal D_Y \otimes_{\cA_Y}^\mathbb{L} \O_Y \cong \O_Y \otimes_\C H_{-\bullet} (\mathfrak p, \C).\]
  More generally,
  \[\cC^\bullet(\cM,\beta) = \mathcal D_Y \otimes_{\mathcal A_Y}^\mathbb{L} \cM \{\beta\} \cong \O_Y \otimes_\C H_{-\bullet}(\mathfrak p, \C_{\differential \mu + \restr{\beta}{\mathfrak p}})\]
  for $\beta \colon \mathfrak g \to \C$ integrable (i.e., of the form $\differential \nu$ for $\nu \colon G\to \C^*$) and any $G$-equivariant line bundle $\cM$ on $Y$ with corresponding character $\mu \colon P \to \C^*$.
  \end{prop}

For the proof, we need the following technical statement on Lie algebroids, which is the Lie algebroid version of the Hochschild--Serre spectral sequence.

\begin{lem} \label{lem:HochschildSerreLieAlgebroids}
  Let $0 \to \mathcal E_1 \to \mathcal E_2 \to \mathcal E_3 \to 0$ be a short exact sequence of locally free Lie algebroids of finite rank on an algebraic variety $Y$ and assume that the anchor map $\mathcal E_1 \to \Theta_Y$ is trivial, i.e., $[\mathcal E_1,\O_Y] = 0$ in $\mathcal U(\mathcal E_1)$. Let $\mathcal U_i \coloneqq \mathcal U(\mathcal E_i)$ denote the corresponding universal envelopping algebras. Let $\mathcal M$ be an $(\mathcal R, \mathcal U_2)$-bimodule. Then, there is a spectral sequence
  \[E^2_{pq} = H^p \mathcal S_{\mathcal R|\mathcal U_3}^\bullet(H^q \mathcal S_{\mathcal R|\mathcal U_1}^\bullet(\mathcal M)) \Longrightarrow H^{p+q} \mathcal S_{\mathcal R|\mathcal U_2}^\bullet(\mathcal M),\]
  where $\cS_{\cR|\mathcal U_i}^\bullet(\cdot) := (\cdot ) \otimes_{\O_Y} \bigwedge_{\O_Y}^{-\bullet} \cE_i$ denotes the Chevalley--Eilenberg complex for right modules over Lie algebroids, see \cite[Definition~2.5]{GS-Duality}.
\end{lem}

\begin{proof}
  This is analogous to the case of the Hochschild--Serre spectral sequence for Lie algebras. Here, the assumption $[\mathcal E_1, \O_Y] = 0$ is needed to ensure that $\mathcal S_{\mathcal R|\mathcal U_1}^\bullet(\mathcal M)$ is a complex of $(\cR,\cU_2)$-modules.
\end{proof}

\begin{proof}[Proof of \cref{prop:HatTOrbit}]
Since $\beta$ is integrable, $\cM\{\beta\}$ underlies a $G$-equivariant line bundle (with character $\mu \cdot \restr{\nu}{P} \colon P \to \C^*$), so we may restrict the proof to the case that $\beta = 0$.

We know that $\mathcal D_Y \otimes_{\mathcal A_Y}^\mathbb L \cM = \mathcal S_{\mathcal D_Y|\mathcal A_Y}^\bullet(\mathcal D_Y \otimes_{\O_Y} \cM)$ by \cite[Lemma 2.12]{GS-Duality}. Since $G$ acts transitively on $Y$, we have a surjection $Z_Y \colon \O_Y \otimes \mathfrak g \twoheadrightarrow \Theta_Y$ of Lie algebroids, inducing a surjection $\mathcal A_Y = \mathcal U(\O_Y \otimes \mathfrak g) \twoheadrightarrow \mathcal U(\Theta_Y) = \mathcal D_Y$. Then
\[\mathcal E \coloneqq \ker(\O_Y \otimes \mathfrak g \twoheadrightarrow \Theta_Y)\]
satisfies $[\O_Y \otimes \mathfrak g,\mathcal E] \subseteq \mathcal E$ and thus is a Lie algebroid ideal in $\O_Y \otimes \mathfrak g$. Note that $\mathcal E$ is locally free of rank $d_G-d_{\hat{X}}$. By definition, we have \begin{equation} \label{eq:OCommutesWithE}
[\mathcal E,\O_Y] = 0.
\end{equation}

By \cref{lem:HochschildSerreLieAlgebroids}, we have a spectral sequence
\[E^2_{pq} = H^p \mathcal S_{\mathcal D_Y|\mathcal D_Y}^\bullet(H^q S_{\mathcal D_Y|\mathcal U(\mathcal E)}^\bullet(\mathcal D_Y \otimes_{\O_Y} \cM)) \Longrightarrow H^{p+q} \mathcal S_{\mathcal D_Y|\mathcal A_Y}^\bullet(\mathcal D_Y \otimes_{\O_Y} \cM).\]

In our case, $\mathcal S_{\mathcal D_Y|\mathcal U(\mathcal E)}^\bullet(\mathcal D_Y \otimes_{\O_Y} \cM) = \mathcal D_Y \otimes_{\O_Y} \cM \otimes_{\O_Y} \bigwedge^{-\bullet} \mathcal E$ is a complex of $(\mathcal D_Y,\mathcal A_Y)$-modules, where the left $\mathcal D_Y$-module structure is given by the action on the first factor, while the right $\mathcal A_Y$-module structure is given by
\[(P \otimes \xi_1 \wedge \dots \wedge \xi_\ell) \cdot \eta = P Z_Y(\eta) \otimes \xi_1 \wedge \dots \wedge \xi_\ell - \sum_{i=1}^\ell P \otimes \xi_1 \wedge \dots \wedge \xi_{i-1} \wedge [\eta,\xi_i] \wedge \xi_{i+1} \wedge \dots \wedge \xi_\ell\]
(for $P \in \mathcal D_Y \otimes_{\cO_Y} \cM$, $\xi_1,\dots,\xi_\ell \in \mathcal E$, $\eta \in \O_Y \otimes \mathfrak g$).
One checks that the right action of $\mathcal E$ on the cohomologies $H^{\ell}(\mathcal D_Y \otimes_{\O_Y} \cM \otimes_{\O_Y} \bigwedge^{\bullet} \mathcal E) \cong \mathcal D_Y \otimes_{\O_Y} H^{-\ell}(\cM \otimes_{\O_Y} \bigwedge^{-\bullet} \mathcal E)$ is trivial, giving rise to a $(\mathcal D_Y,\mathcal D_Y)$-bimodule structure on $\mathcal D_Y \otimes_{\O_Y} H^{-\ell}(\cM \otimes_{\O_Y} \bigwedge^{-\bullet} \mathcal E)$. Notice that the left action is given by left multiplication on the first factor, while the right action is induced from right multiplication on the first factor and a left $\mathcal D_Y$-module structure on $H^{-\ell}(\cM \otimes_{\O_Y} \bigwedge^{-\bullet} \mathcal E)$. Using \cite[Lemma 2.13]{GS-Duality}, we conclude from the spectral sequence above that
\[H^{-\ell}(\mathcal D_Y \otimes_{\cA_Y}^{\mathbb L} \cM) \cong (\mathcal D_Y \otimes_{\O_Y} H^{-\ell}(\cM \otimes_{\O_Y} {\textstyle \bigwedge^{-\bullet} \mathcal E})) \otimes_{\mathcal D_Y}^\mathbb{L} \O_Y \cong H^{-\ell}(\cM \otimes_{\O_Y} {\textstyle \bigwedge^{-\bullet} \mathcal E}).\]
The complex $\cC^\bullet(\cM, 0) \cong \cD_Y \otimes_{\cA_Y}^\mathbb{L} \cM$ is a $G$-equivariant object in $D^b_{qc}(\cD_{Y})$ in the sense mentioned in the proof of \cref{lem:RestToPoint}, as shown in \cite[proof of Proposition~3.3]{GS-Duality}. Therefore, by \cref{eq:EquiModTrivial}, it is isomorphic to $\O_Y \otimes_\C i_0^+ (\cD_Y \otimes_{\cA_Y}^\mathbb{L} \cM)$. Moreover, by \cref{rmk:Formality}, it is a formal object in $D^b_{qc}(\cD_Y)$, so to show the claim, it suffices to see that on the level of cohomologies, $H^{-\ell}i_0^+ (\cD_Y \otimes_{\cA_Y}^\mathbb{L} \cM) \cong H_\ell(\mathfrak p, \C_{\differential \mu})$.

For this, note that the differential of the complex $\cM \otimes_{\O_Y} \bigwedge^{-\bullet} \mathcal E$ commutes with restriction to the fiber over a point $y_0 \in Y$ (this uses \cref{eq:OCommutesWithE}), so
\begin{align*}
i_0^+ H^{-\ell}(\cD_Y \otimes_{\cA_Y}^\mathbb{L} \cM)
\cong i_0^+ H^{-\ell}(\cM \otimes_{\O_Y} {\textstyle \bigwedge^{-\bullet} \mathcal E})
&\cong H^{-\ell}(\restr{\cM}{y_0} \otimes_{\C} {\textstyle \bigwedge^{-\bullet} \restr{\mathcal E}{y_0}}) \\
&\cong H^{-\ell}(\C_{\differential \mu} \otimes_{\C} {\textstyle \bigwedge^{-\bullet} \mathfrak p}) = H_\ell(\mathfrak p,\C_{\differential\mu}),
\end{align*}
where $\mathfrak p \subseteq \mathfrak g$ is the Lie algebra of $P \coloneqq \operatorname{Stab}(y_0) \subseteq G$. Since $i_0$ is non-characteristic with respect to $H^{-\ell}(\cD_Y \otimes_{\cA_Y}^\mathbb{L} \O_Y)$ (since the latter are smooth $\cD_Y$-modules), we obtain
\[H^{-\ell} i_0^+ (\cD_Y \otimes_{\cA_Y}^\mathbb{L} \O_Y) = i_0^+ H^{-\ell}(\cD_Y \otimes_{\cA_Y}^\mathbb{L} \O_Y) \cong H_\ell(\mathfrak p, \C_{\differential\mu}),\]
concluding the proof.
\end{proof}

In order to proceed, we need the following folklore fact from Lie theory; we include a proof for the convenience of the reader.
\begin{lem}\label{lem:HomCohom}
  Let $P \subseteq G$ be a parabolic subgroup.
    Then
  \[H_\bullet(\mathfrak p, \C) \cong H_\bullet(P, \C).\]
\end{lem}

\begin{proof}
  Let $U$ be the unipotent radical of $P$ and let $Q \subseteq P$ be a Levi subgroup, so that $P \cong U \rtimes Q$, in particular $Q$ is reductive, see \cite[25.24]{Milne}. Since $U$ is contractible as a topological space, see \cite[Proposition~14.32]{Milne}, we have
  \[H_\bullet(P, \C) \cong H_\bullet(Q,\C).\]
  Since $Q$ is reductive, the homology of $Q$ can be computed by Lie algebra homology as in \cref{lem:VanishLieHomReductive},
  \[
  H_\bullet(Q,\C) \cong H_\bullet(\mathfrak q, \C).\]

  It remains to see that, on the level of Lie algebra homology, $H_\bullet(\fp,\C) \cong H_\bullet(\mathfrak q,\C)$. For this, consider the  Lie algebra-version of the Levi decomposition $\mathfrak p \cong \mathfrak u \oplus \mathfrak q$,
  where $\mathfrak u$ is  the nilradical  $\mathfrak p$  and $\mathfrak q \subseteq \mathfrak p$ is a suitable Lie subalgebra (in particular, the decomposition $\mathfrak p \cong \mathfrak u \oplus \mathfrak q$ is only an isomorphism of vector spaces  since $\fq$ is not an ideal in $\fp$ in general). From the Hochschild--Serre spectral sequence
  \[
  E^2_{k\ell} = H_k(\mathfrak q, H_\ell(\mathfrak u, \C)) \Longrightarrow H_{k+\ell}(\mathfrak p, \C),
  \]
  we see that it suffices to show that $H_k(\mathfrak q, H_\ell(\mathfrak u, \C)) = 0$ for all $k,\ell$ with $\ell \neq 0$. Note that the action of $\mathfrak q$ on $H_\ell(\mathfrak u, \C)$ here comes from an action of
  $Q \subseteq P$ on each term of the complex $\bigwedge^\bullet \mathfrak u$. Namely, these are the exterior powers of the adjoint $Q$-representation $\Ad \colon Q \to \GL(\mathfrak u)$ (the restriction of $\Ad \colon P \to \GL(\mathfrak p)$ to the ideal $\mathfrak u \subseteq \mathfrak p$ and to the subgroup $Q \subseteq P$). This $Q$-action is compatible with the differential of the complex $\bigwedge^\bullet \mathfrak u$, hence gives a $Q$-action on $H_\ell(\fu,\dC)$ which induces the $\mathfrak q$-action on these spaces used in the spectral sequence above.

    By \cref{lem:VanishingLieHomReductiveGroup}, in order to show that $H_k(\mathfrak q, H_\ell(\fu, \C)) = 0$ for all $k$ and $\ell \neq 0$, it therefore suffices to show that $H_\ell(\mathfrak u,\C)$ has no non-zero $Q$-invariants, or equivalently, no non-zero elements annihilated by $\mathfrak q$.
  For this, we consider the action of the Cartan subalgebra $\mathfrak h$.   We may assume with respect to a weight space decomposition of $\mathfrak g$,
  \[
  \mathfrak g = \mathfrak h \oplus \left(\bigoplus_{\alpha \in \Phi^+} \mathfrak g_\alpha \right)\oplus \left(\bigoplus_{\alpha \in \Phi^+} \mathfrak g_{-\alpha}\right),
  \]
  that the parabolic subalgebra $\mathfrak p$ is given as
  \[\mathfrak p = \mathfrak h \oplus \bigoplus_{\alpha \in \Phi^+} \mathfrak g_\alpha \oplus \bigoplus_{\alpha \in \Phi_I} \mathfrak g_{-\alpha}\]
  for $\Phi_I \coloneqq \Phi^+ \cap \N I$, where $I$ is some subset of the set of simple roots. Then
  \[
  \mathfrak u = \bigoplus_{\alpha \in \Phi^+ \setminus \Phi_I} \mathfrak g_\alpha \qquad \text{and} \qquad \mathfrak q = \mathfrak h \oplus\left( \bigoplus_{\alpha \in \Phi_I} \mathfrak g_\alpha\right) \oplus \left(\bigoplus_{\alpha \in \Phi_I} \mathfrak g_{-\alpha}\right).
  \]
  In particular, $\mathfrak h$ acts on $\mathfrak u$ with positive weights only and therefore, the induced representation $\bigwedge^\ell \mathfrak u$ for $\ell > 0$ decomposes into weight spaces with weights from
    the $\ell$-fold sum
  \[
  \underbrace{\Phi^+ + \dots + \Phi^+}_{\ell \text{ times}}\,\, \not\ni \,\,0.
  \]
  Since the differential of the complex $\bigwedge^{\ell} \fu$ is compatible with the $\fq$-action and therefore in particular with the weight decomposition of the terms, this yields a decomposition into weight spaces of the cohomology groups $H_\ell(\mathfrak u, \C)$ with non-zero weights only. In particular, these cohomology groups do not contain non-zero elements annihilated by $\mathfrak h$ and have in particular no non-zero $Q$-invariants.

  As discussed above, via the Hochschild--Serre spectral sequence, this implies
  \[H_\bullet(\mathfrak p, \C) \cong H_\bullet(\mathfrak q, \C),\]
  which together with the first two observations of the proof implies the claim.
  \end{proof}
Our main result in this section is then the following.
\begin{thm}\label{thm:ComplexOnY}
 Let $Y$ be a $G$-orbit such that the stabilizer $P = \operatorname{Stab}(y_0)$ over a point $y_0 \in Y$ is connected. Moreover, assume that $P$ can be realized as a parabolic subgroup of some connected algebraic group (which might differ from $G$). Then we have $D^b_h(\cD_Y)$-isomorphisms
  \[
    \pi_\dag \cO_G[d_P] \cong  \mathcal D_Y \otimes_{\mathcal A_Y}^{\mathbb L} \O_Y \qquad
    \text{and} \qquad \pi_+ \cO_G \cong  \mathcal D_Y \otimes_{\mathcal A_Y}^{\mathbb L} \omega_Y^\vee.
  \]
              \end{thm}

\begin{proof}
  The first claim follows by putting together the previous results. Namely, we have
  \[
  \cD_Y \otimes_{\cA_Y}^\mathbb{L} \O_Y
  \stackrel{\ref{prop:HatTOrbit}}{\cong} \O_Y \otimes_{\C} H_{-\bullet}(\mathfrak p, \C)
  \stackrel{\ref{lem:HomCohom}}{\cong} \O_Y \otimes_{\C} H_{-\bullet}(P, \C) \cong \O_Y \otimes_\C a_\dag \O_P[d_P] \stackrel{\ref{lem:RestToPoint}}{\cong} \pi_\dag \O_G[d_P].
  \]
    The second claim follows by applying the duality result
  \[\mathbb D(\mathcal D_Y \otimes_{\mathcal A_Y}^{\mathbb L} \O_Y) \cong \mathcal D_Y \otimes_{\mathcal A_Y}^{\mathbb L} \omega_Y^\vee[-d_P]\]
  from \cite[Theorem~3.9]{GS-Duality}.
\end{proof}

As a consequence, using \cref{prop:FunctConstrHatT}, we obtain the following result for the case of the affine cone $\hat X$ over a projective homogeneous space $X = G_0/P_0$ inside a $G_0$-equivariant embedding $X \hookrightarrow \P V$, $V = H^0(X,\Ell)^\vee$, considered with a group action of $G = \C^* \times G_0$, as in \cref{sec:Coloc}.

\begin{thm}\label{theo:ColocFunct}
If $\dim X > 0$ and $\hat X$ is Gorenstein, then we have  $D^b_h(\mathcal{D}_V)$-isomorphisms
$$
\hat{T}(\rho,\hat{X},\gamma) \cong \iota_\dag \pi_\dag \cO_G[d_{P_0}] \cong \iota_\dag \cO_Y \otimes_\dC a_\dag \cO_{P_0}[d_{P_0}]
$$
and
$$
\hat{T}(\rho,\hat{X},0) \cong \iota_+ \pi_+ \cO_G \cong \iota_+ \cO_Y \otimes_\C a_+ \cO_{P_0},
$$
where $\gamma \colon \mathfrak g \to \C$ is as in \cref{rem:GorensteinGamma}.
\end{thm}
\begin{proof}
  This follows from \cref{prop:FunctConstrHatT} and \cref{rem:GorensteinGamma} together with \cref{thm:ComplexOnY} applied to $Y = \hat X \setminus \{0\}$. We observe that the stabilizer $P \subseteq G = \C^* \times G_0$ of a point $y_0 \in \hat X \setminus \{0\}$ is isomorphic to $P_0$ which is connected as follows from the connectedness of $G$ and the fact that $\pi_1(X)=\{1\}$ (see \cite[Remark 4.30]{GRSSW}).
  \end{proof}

 In consequence, the complexes $\hat{T}(\rho,\hat{X},\gamma)$ and $\hat{T}(\rho,\hat{X},0)$ have a Hodge theoretic meaning.
\begin{cornoproof}\label{cor:ThatHodge}
Under the assumptions of  \cref{theo:ColocFunct}, the complexes
$\hat{T}(\rho,\hat{X},\gamma)[-d_{P_0}]$ and $\hat{T}(\rho,\hat{X},0)$ underly
the objects
$\iota_! \pi_! {^p}\dQ_G^H$ and $\iota_* \pi_* {^p}\dQ_G^H$ in $D^b(\MHM(V))$, respectively.
\end{cornoproof}

Recall from \cref{prop:LocCohomCone} that
\[H^k \iota_+ \O_Y \cong \begin{cases}
  H_{\hat X}^{d_V-d_{\hat X}+k} \O_V &\text{if }k \in [0,d_{\hat X}-2] \\[0.25em]
  H_{\{0\}}^{d_V} \O_V & \text{if }k = d_{\hat X}-1 \\[0.25em]
  0 &\text{otherwise.}
\end{cases}\]

Combining this with \cref{theo:ColocFunct}, we obtain that the cohomologies of $\hat{T}(\rho, \hat X, 0)$ are described in terms of the local cohomology of $V$ with respect to $\hat X$ and the homology of the parabolic subgroup $P$.
We can also express the local cohomological defect of $\hat{X}$ in terms of the cohomology groups of $\hat{T}$. Recall that for a subvariety $Z'\subseteq Z$ of a smooth algebraic variety $Z$, the local cohomological defect of $Z$ (which is independent of the embedding $Z'\hookrightarrow Z$) is
$$
\operatorname{lcdef}(Z')\coloneqq\max_i\left\{H^i_{Z'}(Z)\neq 0\right\}-\codim_Z(Z').
$$
\begin{cornoproof}\label{cor-lcdef}
Under the hypotheses of \cref{theo:ColocFunct}, we have
\[H^{-k} \hat T(\rho, \hat X, 0) \cong \big(H_{\{0\}}^{d_V} \O_V \otimes_\C H^{d_P-d_{\hat X}+1-k}(P,\C)\big) \oplus
\bigoplus_{\substack{i+j=d_P-k \\ i \in [0,d_{\hat X}-2] \\ j \in [0,d_P-d_{\hat X}+1]}} H^{d_V-d_{\hat X}+i}_{\hat X} \O_V \otimes_\C H^j(P,\C).\]

In particular, as $\rk H^{-k} T(0) =\dim_\dC i_0^+ H^{-k} \hat{T}(0)$,
\begin{eqnarray*}
\operatorname{lcdef}(\hat X) = d_{\hat X}-1-\min\left\{k \mid  \rk H^{-k} T(0)> \dim H^{d_P-k}(P,\C)\right\}.
\end{eqnarray*}
\end{cornoproof}

The situation of \cref{cor-lcdef} is illustrated by the diagram below.
\begin{center}
\begin{tikzpicture}[x=1.05cm,y=1.05cm, line cap=round, line join=round]

\definecolor{skblue}{RGB}{0,0,0}
\definecolor{skred}{RGB}{190,60,60}

\coordinate (RBL) at (0,0);
\coordinate (RTR) at (2.15,4.0);

\draw[skblue, line width=1.2pt] (RBL) rectangle (RTR);

\path[pattern=north east lines, pattern color=skblue!55]  (0,0) rectangle (0.75,4.0);

\fill[skblue!18] (2.02,0) rectangle (2.15,4.0);
\draw[skblue, line width=1.2pt] (2.02,0) -- (2.02,4.0);

\node[skblue!80] at (1.2,2) {\Large $0$};

\draw[skblue, line width=1.1pt] (-0.35,0) -- (-0.35,4.0);
\draw[skblue, line width=1.1pt] (-0.43,4.0) -- (-0.27,4.0);
\draw[skblue, line width=1.1pt] (-0.43,0.0) -- (-0.27,0.0);

\node[skblue] at (-1,4) {$\SC -(d_{\hat X}-1)$};
\node[skblue] at (-0.8, 0) {$\SC -d_P$};

\draw[skblue, line width=1.1pt, -{Latex[length=2.2mm]}] (-0.10,6.15) -- (3.35,6.15);
\draw[skblue, line width=1.1pt] (0.7,6.02) -- (0.7,6.27);
\draw[skblue, line width=1.1pt] (2.05,6.02) -- (2.05,6.27);

\node[skblue] at (0.8,6.4) {$\SC\text{lcdef}(\hat{X})$};
\node[skblue] at (2.0,6.4) {$\SC d_{\hat X}-1$};
\node[skblue] at (3.55,6.27) {$i$};

\draw[skblue, line width=1.1pt, -{Latex[length=2.2mm]}] (0,0) -- (0,7);

\draw[skred, line width=1.2pt] (-0.55,6.72) -- (3.10,2.92);
\draw[skred, line width=1.2pt] (-0.55,6.32) -- (3.10,2.57);
\draw[skred, line width=1.2pt] (-0.50,5.31) -- (3.10,1.60);

\fill[skred] (0.75,4.00) circle (0.07);   \fill[skred] (2.1,2.63) circle (0.07);
\fill[skblue] (2.1,3.95) circle (0.07);

\draw[skblue, line width=1.1pt] (2.1,3.8) -- (2.55,4.00);

\node[skblue, anchor=west] at (2.60,3.95)
  {\small $H_{\{0\}}^{d_V} \O_V \otimes_\dC H^{d_P+j}(P,\C)$};

\node[skred, anchor=west] at (3,2.90) {\small $H^{0} \,\hat{T}(0)$};
\node[skred, anchor=west] at (3,2.50) {\small $H^{-1} \,\hat{T}(0)$};
\node[skred, anchor=west] at (3,1.53) {\small $H^{d_{\hat X}-1-\text{lcdef}(\hat{X})} \,\hat{T}(0)$};

\node[skblue, anchor=north west] at (1,-0.5)
  {\small $H_{\hat{X}}^{d_V-d_{\hat X}+i}\cO_V \otimes_\C H^{d_P+j}(P,\dC)$};
\draw[skblue, line width=1.1pt, -{Latex[length=2.2mm]}] (2.5,-0.5) -- (0.5,2);

\node[skblue] at  (0.1,7.20) {$j$};

\end{tikzpicture}
\end{center}

Combining \cref{cor-lcdef} with the weight estimations from
\cref{cor:WeightIota},we obtain the following results on the weights of the the cohomologies of $\hat{T}(\rho,\hat{X},0)$ for the case where $P=B$ is a Borel subgroup. In general, due to the non-pureness of the cohomologies of $P$, this estimation becomes more complicated.
\begin{cor}\label{cor:WeightComplexThat}
Under the hypotheses of \cref{theo:ColocFunct}, if $P=B$ is a Borel subgroup, then, for $k \in [0,d_P]$, we have:
$$
H^{-k} \hat T(\rho, \hat X, 0) \cong \left(H_{\{0\}}^{d_V} \O_V\right) \oplus \left(
\bigoplus_{i \in [0, d_{\hat X}-1] \cap [d_{\hat X}-1-k,d_P-k]}
H^{d_V-d_{\hat X}+i}_{\hat X} \O_V\right) .
$$
Moreover, we have
\[
\textup{weights}\left(H^{-k} \hat{T}(0)\right)
\subseteq\left[d_{\hat X}+2d_P+1-2k-\operatorname{lcdef}(\hat X),d_{\hat X}+2d_P+1-2k\right] \cup \{2d_P-2k+1\},
\]
and more precisely
\begin{align*}
\textup{weights}\left(H^{-k} \hat{T}(0)\right)
\subseteq&\left\{d_{\hat X}+2d_P+1-2k-i \mid i \in ([0,\, \operatorname{lcdef}(\hat X)] \cup \{d_{\hat X} -1\}) \cap [d_{\hat X}-1-k,\, d_P-k] \right\}
\end{align*}
with
\[Gr^W_{d_{\hat X}+2d_P+1-2k-i}H^{-k} \hat T(0) \cong
\begin{cases}
  \coker(H^0\iota_\dag \O_{\hat X} \to H^0\iota_+ \O_{\hat X}) &\text{if } i = 0 \\[0.25em]
  H^{d_V-d_{\hat X}+i}_{\hat X} \O_V \oplus \iota_{\dag +} \O_{\hat X \setminus \{0\}} &\text{if } i = 1 \text{ and } k > d_{\hat X}-2, \\[0.25em]
  H^{d_V-d_{\hat X}+i}_{\hat X} \O_V &\text{if } i = 1 \text{ and } k = d_{\hat X}-2, \\[0.25em]
  H^{d_V-d_{\hat X}+i}_{\hat X} \O_V &\text{if } i \in [2,\operatorname{lcdef}(\hat X)], \\[0.25em]
  H^{d_V}_{\{0\}} \O_V &\text{if } i = d_{\hat{X}}-1.
\end{cases}\]

In particular, for $k \in [0, d_{\hat X}-2-\operatorname{lcdef}(\hat X)]$, the mixed Hodge module associated with $H^{-k} \hat T(\rho, \hat X, 0)$ by \cref{cor:ThatHodge} is pure of weight $2d_P-2k+1$.

\end{cor}
\begin{proof}
We deduce from the second statement of \cref{theo:ColocFunct} that
$$
H^{-k} \hat{T}(0) = \bigoplus_{i+j=-k} H^i\iota_+\cO_Y \otimes_\dC H^j a_+\cO_P
=
\bigoplus_{i+j=-k} H^i\iota_+\cO_Y \otimes_\dC H^{d_P+j}(P,\dC).
$$
If $P=B$ is a Borel subgroup, then its Levi factor $Q$ is a torus. Therefore, $H^j(B,\dQ)\cong H^j(Q,\dQ)$ is one-dimensional for each $j$ and has a pure Hodge structure of weight $2j$ (see, e.g., \cite[Example 3.9]{Batyrev5}). Now the statement follows from \cref{cor:WeightIota}.
\end{proof}
As an easy consequence, for the case $P=B$, the weight filtration on the cohomology groups of the complex $\hat{T}(0)$ allows us obtain the various local cohomology groups of the cone $\hat{X}$.
\begin{cor}
  In the situation of \cref{cor:WeightComplexThat} (i.e., $P=B$), all local cohomologies $H_{\hat X}^\ell \O_V$ are graded pieces of the weight filtration of some cohomology group of $\hat T(\rho, \hat X, 0)$.
\end{cor}

\begin{rem}\label{rem:DuBois}
In a series of papers, Popa and Musta\c{t}\u{a} (see, e.g., \cite{PopaMustata-LocCohom}) have developed the Hodge theory of local cohomology groups of singular subvarieties. In particular, they have shown (see \cite[Proposition 5.5]{PopaMustata-LocCohom}) that if $k:Z \hookrightarrow \widetilde{Z}$ is the embedding of a reduced subvariety $Z$ into a smooth ambient variety $\widetilde{Z}$, then there is an isomorphism
$$
\underline{\Omega}^p_Z \cong
{R\cH\!}om_{\cO_{\widetilde{Z}}}
\left(
\Gr^F_{p-d_{\widetilde{Z}}}\textup{DR}_{\widetilde{Z}} k_! k^!
\mbqH_{\widetilde{Z}}[d_{\widetilde{Z}}], \omega_{\widetilde Z}
\right)[p]
$$
where $\underline{\Omega}^p_Z$ denotes the $p$-th du Bois complex of $Z$.
The complex of $\cD_{\widetilde{Z}}$-modules underlying $k_! k^! \mbqH_{\widetilde{Z}}[d_{\widetilde{Z}}]$ is $R\Gamma_Z \cO_{\widetilde{Z}}$. For the case $\widetilde{Z}=V$, $Z=\hat{X}$, we have seen in \cref{prop:LocCohomCone} that the cohomologies of this complex are given, up to degree $d_V-2$ by the groups $H^i \iota_+ \cO_{\hat{X}\backslash\{0\}}$, and the latter are the building blocks of the complex $\hat{T}(\rho,\hat{X},0)$, according to \cref{theo:ColocFunct}. Therefore, the knowledge of the Hodge filtration on $\hat{T}(\rho,\hat{X},0)$ would give a rather precise description of the $p$-th du Bois complex of the cone $\hat{X}$, which might be of interest for some questions on the underlying homogeneous space $X$. Moreover, the Hodge filtration on $\hat{T}(\rho,\hat{X},0)$ determines that of its dual complex
complex, which is $\hat{T}(\rho, \hat{X},\gamma)$ by the duality results from \cite{GRSSW}. The top cohomology of the latter is nothing but the tautological system $\hat{\tau}(\rho, \hat{X},\gamma)$, for which the knowledge of the Hodge filtration would be very useful for applications in mirror symmetry, along the lines of \cite{RS12, RS20}. We leave a thorough discussion of these matters to a subsequent paper.
\end{rem}

\begin{rem} \label{rem:NonZeroTHat}
  With the arguments in this paper, we also directly see that for $\hat X$ Gorenstein, the only integrable parameters $\beta\colon \mathfrak g \to \C$  leading to non-zero $\hat T(\rho, \hat X, \beta)$ are in fact only the two parameters $0$ and $\gamma$ as in \cref{rem:GorensteinGamma}, the same non-vanishing behaviour as for $H^0 \hat T(\rho, \overline X, \beta)$. This can be argued as follows:
  On one hand, $\hat T(\rho, \hat X, \beta)$ is colocalized for $\beta \neq 0$. On the other hand, by duality \cite[Theorem~4.6]{GS-Duality} for $\hat X$ Gorenstein, we have $\mathbb D\hat T(\rho, \hat X, \beta) \cong \hat T(\rho, \hat X, \gamma-\beta)[-d_P]$, so accordingly, \cref{prop:FunctConstrHatT} also shows that $\hat T(\rho, \hat X, \beta)$ is localized for $\beta \neq \gamma$.
  In particular, if $\beta$ is neither $0$ nor $\gamma$, then $\hat T(\rho, \hat X, \beta)$ is both localized and colocalized. By the description of the restriction to $\hat X \setminus \{0\}$ in \cref{prop:HatTOrbit}, this means
  \begin{equation} \label{eq:locAndColoc}
  \hat T(\rho, \hat X, \beta) \cong \iota_\star \O_{\hat X \setminus 0} \otimes_\C H_{-\bullet}(\mathfrak p, \C_{\gamma-\beta}) \qquad \text{ for both $\star=+$ and $\star=\dag$}.
  \end{equation}
  However, the cohomological amplitude of $\iota_+ \O_{\hat X \setminus \{0\}}$ is $[0,d_{\hat X}-1]$, while that of $\iota_\dag \O_{\hat X \setminus \{0\}}$ is $[d_{\hat X}-1,0]$, and the cohomologies at the endpoints of the intervals are non-zero in each case (by \cref{prop:LocCohomCone}). Since $d_{\hat X} - 1 \geq 1$, two descriptions of $\hat T(\rho, \hat X, \beta)$ in \cref{eq:locAndColoc} force $H_{-\bullet}(\mathfrak p, \C_{\gamma-\beta}) = 0$ (otherwise, one obtains a contradiction by considering the highest non-vanishing cohomology of $\hat T(\rho, \hat X, \beta)$). Hence, $\hat T(\rho, \hat X, \beta) = 0$ for $\beta \notin \{0,\gamma\}$.
\end{rem}

\bibliographystyle{amsalpha}

\begin{thebibliography}{GRSSW26}

\bibitem[Bat93]{Batyrev5}
Victor~V. Batyrev, \emph{Variations of the mixed {H}odge structure of affine
  hypersurfaces in algebraic tori}, Duke Math. J. \textbf{69} (1993), no.~2,
  349--409.

\bibitem[Bat26]{Batavia}
Manav Batavia, \emph{Vanishing of local cohomology in unramified mixed
  characteristic}, 2026,
  \href{https://arxiv.org/abs/2602.22191}{arXiv:math/2602.22191}.

\bibitem[BBD82]{BBD82}
A.~A. Beilinson, J.~Bernstein, and P.~Deligne, \emph{Faisceaux pervers},
  Analysis and topology on singular spaces, {I} ({L}uminy, 1981), Ast\'erisque,
  vol. 100, Soc. Math. France, Paris, 1982, pp.~5--171. \MR{751966}

\bibitem[Bot57]{Bot57}
Raoul Bott, \emph{Homogeneous vector bundles}, Ann. of Math. (2) \textbf{66}
  (1957), 203--248.

\bibitem[CE48]{ChevalleyEilenberg}
Claude Chevalley and Samuel Eilenberg, \emph{Cohomology theory of {L}ie groups
  and {L}ie algebras}, Trans. Amer. Math. Soc. \textbf{63} (1948), 85--124.

\bibitem[Che99]{Che99}
Sophie Chemla, \emph{A duality property for complex {L}ie algebroids}, Math. Z.
  \textbf{232} (1999), no.~2, 367--388.

\bibitem[Dim04]{Dimca}
Alexandru Dimca, \emph{Sheaves in topology}, Universitext, Springer-Verlag,
  Berlin, 2004.

\bibitem[Gai05]{Gaitsgory}
Dennis Gaitsgory, \emph{Geometric representation theory}, Lectures notes,
  Harvard university, available at
  \url{https://people.mpim-bonn.mpg.de/gaitsgde/267y/catO.pdf}, 2005.

\bibitem[GLS25]{SabbahGarciaLopez}
Ricardo Garc\'ia~L\'opez and Claude Sabbah, \emph{Hodge-{L}yubeznik numbers},
  C. R. Math. Acad. Sci. Paris \textbf{363} (2025), 213--221.

\bibitem[GRSSW26]{GRSSW}
Paul Görlach, Thomas Reichelt, Christian Sevenheck, Avi Steiner, and Uli
  Walther, \emph{{Tautological systems, homogeneous spaces and the holonomic
  rank problem}}, {J. \'Ec. Polytech. Math.} \textbf{13} (2026), 519--591.

\bibitem[GS25]{GS-Duality}
Paul Görlach and Christian Sevenheck, \emph{Duality theory of tautological
  systems}, arXiv 2510.01980, 2025.

\bibitem[Hal15]{Hall}
Brian Hall, \emph{Lie groups, {L}ie algebras, and representations}, second ed.,
  Graduate Texts in Mathematics, vol. 222, Springer, Cham, 2015, An elementary
  introduction.

\bibitem[Har68]{Hartshorne-CDAV}
Robin Hartshorne, \emph{Cohomological dimension of algebraic varieties}, Ann.
  of Math. (2) \textbf{88} (1968), 403--450.

\bibitem[HL90]{HunekeLyubeznik}
C.~Huneke and G.~Lyubeznik, \emph{On the vanishing of local cohomology
  modules}, Invent. Math. \textbf{102} (1990), no.~1, 73--93.

\bibitem[HTT08]{Hotta}
Ryoshi Hotta, Kiyoshi Takeuchi, and Toshiyuki Tanisaki,
  \emph{{$\mathcal{D}$}-modules, perverse sheaves, and representation theory},
  Progress in Mathematics, vol. 236, Birkh\"auser Boston Inc., Boston, MA,
  2008, Translated from the 1995 Japanese edition by Takeuchi.

\bibitem[Mil17]{Milne}
J.~S. Milne, \emph{Algebraic groups}, Cambridge Studies in Advanced
  Mathematics, vol. 170, Cambridge University Press, Cambridge, 2017, The
  theory of group schemes of finite type over a field.

\bibitem[MMW05]{MMW}
Laura~Felicia Matusevich, Ezra Miller, and Uli Walther, \emph{Homological
  methods for hypergeometric families}, J. Amer. Math. Soc. \textbf{18} (2005),
  no.~4, 919--941.

\bibitem[MP22]{PopaMustata-LocCohom}
Mircea Musta\c{t}\u{a} and Mihnea Popa, \emph{Hodge filtration on local
  cohomology, {D}u {B}ois complex and local cohomological dimension}, Forum
  Math. Pi \textbf{10} (2022), Paper No. e22, 58.

\bibitem[NMS19]{NarvaezSevenheck}
Luis Narv\'aez~Macarro and Christian Sevenheck, \emph{Tautological systems and
  free divisors}, Adv. Math. \textbf{352} (2019), 372--405.

\bibitem[Ogu73]{Ogus-LCDAV}
Arthur Ogus, \emph{Local cohomological dimension of algebraic varieties}, Ann.
  of Math. (2) \textbf{98} (1973), 327--365.

\bibitem[OV90]{Onishchik-Vinberg}
Arkadij~L. Onishchik and Ernest~B. Vinberg, \emph{Lie groups and algebraic
  groups}, Springer Series in Soviet Mathematics, Springer-Verlag, Berlin,
  1990, Translated from the Russian and with a preface by D. A. Leites.

\bibitem[PS73]{PeskineSzpiro}
C.~Peskine and L.~Szpiro, \emph{Dimension projective finie et cohomologie
  locale. {A}pplications \`a la d\'{e}monstration de conjectures de {M}.
  {A}uslander, {H}. {B}ass et {A}. {G}rothendieck}, Inst. Hautes \'{E}tudes
  Sci. Publ. Math. (1973), no.~42, 47--119.

\bibitem[RS17]{RS12}
Thomas Reichelt and Christian Sevenheck, \emph{Non-affine {L}andau-{G}inzburg
  models and intersection cohomology}, Ann. Sci. \'Ec. Norm. Sup\'er. (4)
  \textbf{50} (2017), no.~3, 665--753 (2017).

\bibitem[RS20]{RS20}
\bysame, \emph{Hypergeometric {H}odge modules}, Algebr. Geom. \textbf{7}
  (2020), no.~3, 263--345.

\bibitem[RSW21]{RSW}
Thomas Reichelt, Morihiko Saito, and Uli Walther, \emph{Dependence of
  {L}yubeznik numbers of cones of projective schemes on projective embeddings},
  Selecta Math. (N.S.) \textbf{27} (2021), no.~1, Paper No. 6, 22.

\bibitem[RW19]{ReichWalth-Duco}
Thomas Reichelt and Uli Walther, \emph{Gau\ss -{M}anin systems of families of
  {L}aurent polynomials and {$A$}-hypergeometric systems}, Comm. Algebra
  \textbf{47} (2019), no.~6, 2503--2524.

\bibitem[Sal24]{SalamonLieGroups}
Dietmar Salamon, \emph{{Notes on complex Lie groups}}, available at
  \url{https://people.math.ethz.ch/~salamon/PREPRINTS/cx-lie.pdf}, 2024.

\bibitem[Ste19]{AviDualProjRestrGKZ}
Avi Steiner, \emph{Dualizing, projecting, and restricting {GKZ} systems}, J.
  Pure Appl. Algebra \textbf{223} (2019), no.~12, 5215--5231.

\bibitem[SW09]{SW09}
Mathias Schulze and Uli Walther, \emph{Hypergeometric {D}-modules and twisted
  {G}au\ss -{M}anin systems}, J. Algebra \textbf{322} (2009), no.~9,
  3392--3409.

\bibitem[vdB99]{vanDenBergh}
Michel van~den Bergh, \emph{Local cohomology of modules of covariants}, Adv.
  Math. \textbf{144} (1999), no.~2, 161--220.

\bibitem[Wei94]{Weibel}
Charles~A. Weibel, \emph{An introduction to homological algebra}, Cambridge
  Studies in Advanced Mathematics, vol.~38, Cambridge University Press,
  Cambridge, 1994.

\bibitem[Zha26]{Zhang-second}
Wenliang Zhang, \emph{The second vanishing theorem for local cohomology
  modules}, Math. Res. Lett. \textbf{32} (2026), no.~6, 2039--2062.

\end{thebibliography}
\providecommand{\bysame}{\leavevmode\hbox to3em{\hrulefill}\thinspace}
\providecommand{\MR}{\relax\ifhmode\unskip\space\fi MR }
\providecommand{\MRhref}[2]{  \href{http://www.ams.org/mathscinet-getitem?mr=#1}{#2}
}
\providecommand{\href}[2]{#2}

\end{document}